\documentclass[12pt]{article}
\usepackage[utf8]{inputenc}
\usepackage[english]{babel}
\usepackage{amssymb,amsfonts,amsthm}
\usepackage{amsmath}

\usepackage{lineno,hyperref}

\usepackage{graphicx}

\usepackage{color}

\numberwithin{equation}{section}

\newcommand{\R}{\mathbb{R}}
\newcommand{\C}{\mathcal{C}}
\newcommand{\n}{\mathbf{n}}

\newcommand{\CL}{\mathcal{L}}
\newcommand{\CB}{\mathcal{B}}

\def\Xint#1{\mathchoice
    {\XXint\displaystyle\textstyle{#1}}%
     {\XXint\textstyle\scriptstyle{#1}}%
     {\XXint\scriptstyle\scriptscriptstyle{#1}}%
     {\XXint\scriptstyle\scriptscriptstyle{#1}}%
	\!\int}
\def\XXint#1#2#3{{\setbox0=\hbox{$#1{#2#3}{\int}$}
	\vcenter{\hbox{$#2#3$}}\kern-.5\wd0}}

\newtheorem{theorem}{Theorem}[section]

\newtheorem{thm}{Theorem}[section]
\newtheorem{lemma}[thm]{Lemma}
\newtheorem{prop}[thm]{Proposition}
\newtheorem{cor}[thm]{Corollary}
\newtheorem{remark}[thm]{Remark}

\usepackage{babelbib}

\title{Nonstationary Venttsel problems with discontinuous data}
\author{D. E. Apushkinskaya, A. I. Nazarov, D. K. Palagachev,\\ and L. G. Softova }
\date{\today}

\begin{document}

\maketitle

\hfill {\it Dedicated to  Professor Vsevolod A. Solonnikov} 

\hfill {\it on the occasion on his 90th anniversary, with admiration}

\section{Introduction}\label{sec1}

Let $\mathcal{L}$ be a second-order linear elliptic operator defined over 
 a bounded and sufficiently smooth domain  $\Omega\subset \R^n$. 
In his pioneer work \cite{V59},  A.D.~Venttsel found  the most general boundary conditions which restrict $\mathcal{L}$  to an  infinitesimal generator of a Markov process in $\Omega$.   These conditions are given  in terms of a second-order integro-differential  operator $\mathcal{V}=\mathcal{B}+\mathcal{I}$, where the differential (\textit{local}) term $\mathcal{B}$ describes diffusion, drift, absorption, reflection and viscosity along the boundary, while the integral (\textit{non-local}) operator $\mathcal{I}$ corresponds to jumps phenomena on the boundary. A boundary value problem for the operator $\mathcal{L}$ with boundary conditions given by $\mathcal{V}$ takes the name of the \textit{Venttsel boundary value problem.} 

The Venttsel problems, even if originally arising in the theory of stochastic processes,  appear in various branches of technology, industry, engineering, financial mathematics, etc. (see  for instance the references in \cite{AN00}, \cite{CFGGGOR09}, and \cite{ANPS21}).   

In the last decades, much attention has been paid to regularity and solvability issues regarding  boundary value problems for elliptic and parabolic operators with discontinuous coefficients. The essential base for all this research was given by the novel paper \cite{CFL93} where the Calder\'on--Zygmund theory of second-order elliptic operators with \textit{continuous} principal coefficients has been extended to the \textit{discontinuous} situation of operators with principal coefficients having \textit{vanishing mean oscillation.} Regularity and strong solvability theories have been developed for the Dirichlet, Neumann, and Robin boundary value problems both for elliptic and parabolic operators (see \cite{BC93,MPS00,Kr08,DK11} for instance). 

In \cite{ANPS21} regularity and strong solvability in the framework of Sobolev spaces have been obtained for linear and quasilinear elliptic operators subject to  \textit{local} Venttsel boundary conditions in the case when both the interior and the boundary differential operators have \textit{VMO} principal coefficients. It is also important to notice that the assumptions on the lower-order coefficients in \cite{ANPS21} are optimal in terms of the Lebesgue and Orlicz spaces.

In the present article, we extend these  results to parabolic discontinuous Venttsel problems in composite Sobolev spaces, generalizing this way the results about parabolic Venttsel problems for operators with continuous principal coefficients proven in \cite{AN95}.
 The differential operator $\mathcal{L}$ defined in a cylindrical domain and the boundary Venttsel operator $\mathcal{B}$ are supposed to be uniformly parabolic second-order operators both with principal coefficients that are \textit{only measurable} with respect to the time-variable $t$ and have \textit{vanishing mean oscillation} with respect to the spatial variables $x$. In particular, these include operators with principal coefficients depending \textit{only} on the time-variable $t$ and the fact that these may be \textit{merely} measurable in $t$ is of great importance in the theory of stochastic processes (see \cite{Kr07a,Kr07b,Kr08}). Regarding the lower-order coefficients of the operators considered, these are taken in suitable Lebesgue/Orlicz spaces which turn out to be optimal (for possible generalizations, see Remark~\ref{rem-Krylov}).

For the linear Venttsel problem considered we prove coercive \textit{a~priori} estimates and strong solvability in the composite parabolic Sobolev spaces $W^{2,1}_p(\mathcal{Q}_T)\cap W^{2,1}_q(\Gamma_T)$ for all admissible values of $p$ and $q$. Our proofs rely on the techniques of \cite{AN95} combined with the ideas used in \cite{ANPS21} and the results of \cite{DK11} on Cauchy--Dirichlet problems for linear parabolic equations with partially \textit{VMO} coefficients. Precisely, in order to derive the \textit{a~priori} estimate we start with a bound over the lateral boundary of the cylinder for the solution of an ``autonomous'' Venttsel problem that does not contain the directional derivative $\partial_{\mathbf{n}}u$ of the solution along the normal to the boundary. Further, fine interpolations between Sobolev spaces allow to estimate $\partial_{\mathbf{n}}u$ by means of itself and this leads, through the local estimates from \cite{DK11}, to a coercive estimate for the solution over the lateral boundary. Once having that, the solution of the Venttsel problem can be viewed as a solution to Cauchy--Dirichlet problem and the bounds from \cite{DK11} combined with fine trace theorems and perturbation techniques lead to the coercive estimate for the solution to the Venttsel problem. With this estimate in hand, the strong solvability in Sobolev spaces is reached by the method of continuity.

In the second part of the paper, we obtain strong solvability of \textit{quasilinear} parabolic equations subject to \textit{quasilinear} Venttsel boundary conditions. The principal parts  of the both interior and the boundary operators are allowed to be \textit{only measurable} in time and \textit{VMO} in the space variables, these depend nonlinearly on the solution but, due to technical restrictions, are \textit{independent} of its spatial gradient. The lower-order terms instead allow unbounded $(x,t)$-singularities controlled in Lebesgue or Orlicz norms and support the optimal (Bernstein-type)
quadratic gradient growth. The strong solvability of the quasilinear Venttsel problem is reached by means of the Leray--Schauder fixed point theorem, and to apply it we need a series of \textit{a~priori} estimates for all possible solutions of a family of Venttsel problems. The bounds for the $L^\infty$ and the H\"older norms of the solutions follow easily from \cite{AN95a,AN95,AN96}, and the great effort here is concentrated on the estimates for the $L^{2(n+2)}$-norm of the \textit{spatial} gradient inside the cylinder and the $L^{2(n+1)}$-norm of the \textit{tangential} gradient over the lateral boundary. In this step, we employ a homotopy-type argument due to Amann and Crandall (\cite{AC78}) suitably adapted to the Venttsel boundary conditions (see \cite{Sof03} for the cases of Dirichlet or Robin problems).

Some open problems and directions for further research are provided at the end of the paper.

\section{Auxiliary results}\label{sec2}

\subsection{Basic notation}
Throughout the paper we adopt the following notations:

\noindent
$x=(x',x_n)=(x_1,\dots,x_{n-1},x_n)\in \R^n$, $|x|$ is the Euclidean norm of $x$.

\noindent
$(x,t)\in \R^{n+1}$.

\noindent
$\R^n_+=\{x\in \R^n\colon\  x_n>0\}$.

\noindent
$\Omega$ is a bounded domain in $\R^n$, $n\ge2$, with closure $\overline{\Omega}$ and  boundary $\partial\Omega$.

\noindent
$\n(x)=\big(\n_1(x),\ldots,\n_n(x)\big)$ is the unit vector of the outward normal to $\partial \Omega$ at the point $x$.

\noindent
$\mathcal{Q}_T=\Omega\times(0,T)$ is a cylinder in $\R^{n+1}$; $\Gamma_T=\partial\Omega\times(0,T)$ is its lateral boundary and $\partial_P\mathcal{Q}_T=\Gamma_T\cup \big\{
\overline{\Omega}\times\{0\}\big\}$ is its parabolic boundary.

\noindent
$B_\rho(x^0)$ is the open ball in $\R^n$ centered at $x^0$ and of radius $\rho$,  $B_\rho=B_\rho(0)$; $B'_\rho=B_\rho\cap \{x_n=0\}$.

\noindent
$Q_\rho(x^0;t^0)=B_\rho(x^0)\times (t^0-\rho^2,t^0)$ is the standard parabolic cylinder.

For a measurable set $E$, $|E|$ stands for its Lebesgue measure in the corresponding dimension. Further, $L_p(E)$ is the standard Lebesgue space with the norm $\|\cdot\|_{p,E}$, while
$$
\Xint -\limits_E f=\frac 1{|E|}\int\limits_E f
$$
is the mean value of the function $f$ over $E$. In the sequel, we use the notation
$f_{\pm}=\max\{\pm f;0\}$.

The  indices $i,j$ run from $1$ to $n$ 
and the standard summation convention on the  repeated indices is adopted.

\noindent
$D_i$ stands for the operator of  differentiation with respect
to  $x_i$.

\noindent
$Du=( D'u,D_n u) =(D_1u,\ldots,D_{n-1}u,D_nu)$ is the spatial gradient of $u$.

\noindent
$\partial_t$ is the operator of  differentiation with respect to $t$.

\noindent 
$d_i$ denotes the tangential differential operator on $\partial \Omega$ given by
$$
d_i=D_i-\n_i \n_j D_j.
$$

\noindent
$du=(d_1u,\ldots,d_nu)$ is the tangential gradient of $u$ on $\partial \Omega$ while $\partial_{\mathbf{n}}u$ stands for the normal derivative of $u$.


$W^{2,1}_p(\mathcal{Q}_T)$ is the parabolic Sobolev space equipped with the norm
$$
\|u\|_{W^{2,1}_p(\mathcal{Q}_T)}=\|\partial_tu\|_{p,\mathcal{Q}_T}+\|D(Du)\|_{p,\mathcal{Q}_T}+\|u\|_{p,\mathcal{Q}_T}.
$$
In a similar way, we introduce the spaces $W^{2,1}_q(\Gamma_T)$.
\smallskip

We set further $V_{p,q}(\mathcal{Q}_T)=W^{2,1}_p(\mathcal{Q}_T)\cap W^{2,1}_q(\Gamma_T)$ for the collection of all $W^{2,1}_p(\mathcal{Q}_T)$ functions that have traces in 
$W^{2,1}_q(\Gamma_T)$, with the norm 
$$
\|u\|_{V_{p,q}(\mathcal{Q}_T)}=\|u\|_{W^{2,1}_p(\mathcal{Q}_T)}+\|u\|_{W^{2,1}_q(\Gamma_T)}.
$$

$\C^{1,1}$ is the space of functions with Lipschitz continuous first-order derivatives. Further, 
$\partial\Omega\in  \C^{1,1}$ means that there exists an $\mathfrak{R}_0>0$ such that for each $x^0 \in \partial\Omega$ the set $\partial\Omega \cap B_{\mathfrak{R}_0}(x^0)$, in a suitable Cartesian coordinate system, is the
graph $y_n=f(y')$ of a function $f\in \C^{1,1}$.
Saying that a given constant depends on 
\textit{the properties of $\partial\Omega$} means dependence on $\mathfrak{R}_0$, on the $\C^{1,1}$-norms of the 
diffeomorphisms that flatten 
locally $\partial\Omega$  and on the area of $\partial\Omega$.

We use the letters $C$ and $N$ (with or without indices) to denote various constants. To indicate that $C$ depends on some parameters, we list these in parentheses: $C(\dots)$. 

Finally, we set $\frac{0}{0}=0$, if such an uncertainty occurs.

\subsection{Partial VMO functions}

Following \cite{Kr08}, for a locally integrable function $a\colon\ \R^{n+1}\to\R$, we define the \textit{mean oscillation of $a$ with respect to $x$} over the cylinder $Q_\rho(x,t)$ by
\begin{equation}\label{osc-x}
\mathrm{osc}_x\big(a,Q_\rho(x,t)\big)=
\Xint {\ \ -}\limits_{t-\rho^2}^t 
\Xint {\ \ \ -}\limits_{B_\rho(x)}
\Xint {\ \ \ -}\limits_{B_\rho(x)} \big|a(y;\tau)-a(z,\tau)\big|\,dydzd\tau
\end{equation}
and set
\begin{equation}\label{VMO-mod}
a^{\#(x)}_R=\sup_{(x,t)\in\R^{n+1}} \sup_{\rho\leq R}\, \mathrm{osc}_x\big(a,Q_\rho(x,t)\big).
\end{equation}

The function $a$ is said to be of \textit{bounded mean oscillation with respect to $x$, $a\in \textit{BMO}_x$,} if $a^{\#(x)}_R$ is bounded for all $R>0$, while 
$a$ is of \textit{vanishing mean oscillation with respect to $x$, $a\in \textit{VMO}_x$,} when $\lim_{R\to0}a^{\#(x)}_R=0$ and $a^{\#(x)}_R$ is referred as $\textit{VMO}_x$-modulus of $a$.

In case of bounded cylinder $\mathcal{Q}_T$,
 the spaces $\textit{BMO}_x(\mathcal{Q}_T)$  and $\textit{VMO}_x(\mathcal{Q}_T)$   are defined in a similar manner by replacing $Q_\rho$ above by the respective intersections with $\mathcal{Q}_T$. Moreover, in case of smooth $\partial\Omega$ the spaces $\textit{BMO}_x(\Gamma_T)$  and $\textit{VMO}_x(\Gamma_T)$ are defined naturally  by considering surface integral oscillations over $Q_\rho(x,t)\cap\Gamma_T$ with $(x,t)\in\Gamma_T$.

\medskip 

Some remarks regarding the above definitions are in order.

Setting 
\begin{equation}\label{osc}
\mathrm{osc}\big(a,Q_\rho(x,t)\big)=
\Xint {\quad \ -}\limits_{Q_\rho(x,t)} 
\Xint {\quad \ -}\limits_{Q_\rho(x,t)} \big|a(y;\tau)-a(z,\theta)\big|\,dyd\tau\, dzd\theta,
\end{equation}
the John--Nirenberg space $\textit{BMO}$ (\cite{JN61})
and the Sarason class $\textit{VMO}$ (\cite{Sar75}) are defined as the collections of those functions $a\colon\ \R^{n+1}\to \R$ for which the quantity
$$
a^{\#}_R=\sup_{(x,t)\in\R^{n+1}} \sup_{\rho\leq R}\, \mathrm{osc}\big(a,Q_\rho(x,t)\big)
$$
is, respectively, bounded or vanishing as $R\to0$.

Direct computations show that
$$
\Xint {\quad \ -}\limits_{Q_\rho(x,t)} 
\big| a(y;\tau)-a_{Q_\rho(x,t)}\big| dyd\tau \leq 
\mathrm{osc}\big(a,Q_\rho(x,t)\big),
$$
$$
\Xint {\quad \ -}\limits_{Q_\rho(x,t)} 
\big| a(y;\tau)-a_{B_\rho(x)}(\tau)\big| dyd\tau \leq 
\mathrm{osc}_x\big(a,Q_\rho(x,t)\big),
$$
where $a_{Q_\rho(x,t)}=\!\!\Xint {\quad \ \ -}\limits_{Q_\rho(x,t)}  a(y;\tau)dyd\tau$ and
$a_{B_\rho(x)}(\tau)=\!\!\Xint {\quad \,  -}\limits_{B_\rho(x)} a(y;\tau)dy$.
\medskip

This way, $a\in\textit{VMO}$ asks a sort of ``integral continuity'' of $a$, measured in $\textit{VMO}$, with respect to \textit{all variables} $(x,t)$, while $a\in\textit{VMO}_x$ requires that ``continuity'' only with respect to the spatial variables $x$, allowing $a$ to be \textit{merely measurable in $t$.} This is clearly seen in case $a(x,t)$ does not depend on $x$ when $a(t)\in\textit{VMO}_x$ automatically, while $a\in\textit{VMO}$ is not guaranteed generally.

\subsection{An extension result}

The next statement is a generalization of Theorem~6.1 in \cite{AN95}.

\begin{thm}\label{extension-theorem}
Let $\partial\Omega \in \C^{1,1}$ and let the exponents $p$ and $q$ satisfy
$$
1\leq p \leq q\,\frac{n+2}{n+1},\qquad q>1.
$$
Then there exists an extension operator
$$
\Pi\colon\ W^{2,1}_q(\Gamma_T) \rightarrow W^{2,1}_p(\mathcal{Q}_T)
$$
such that
\begin{equation} \label{6.1-AN95}
\|\Pi u\|_{W^{2,1}_p(\mathcal{Q}_T)} \le N_0 \|u\|_{W^{2,1}_q(\Gamma_T)},
\end{equation}
where $N_0$ depends only on $p,q$ and the properties of $\partial\Omega$.
\end{thm}
\begin{proof} 
It is easy to see that it suffices to prove the theorem when $p=q\,\frac{n+2}{n+1}$.

\textit{Step 1.} We start with a  procedure constructing an extension operator from a flat boundary surface to a boundary strip that acts continuously from the space $\hat{W}\vphantom{W}^{2,1}_q(B_R'\times (0,T))$ into the space $W^{2,1}_p(B_R'\times (0,T) \times (0,R))$.
Here $\hat{W}\vphantom{W}^{2,1}_q(B_R'\times (0,T))$ stands for the closure of the set of smooth functions vanishing in a neighborhood of lateral boundary with respect  to the norm in $W^{2,1}_q$, and we assume that the function $u$ is extended as zero to the whole strip $\R^{n-1}\times (0,T)$.

Now the successive application of some statements from \cite{BIN75} yields:

1) embedding \cite[Theorem 18.12]{BIN75}
$$
W^{2,1}_q(\R^{n-1}\times (0,T)) \rightarrow \mathbf{B}^{2-\frac{1}p,1-\frac{1}{2p}}_{p,p}(\R^{n-1}\times (0,T));
$$

2) extension \cite[Theorem 18.13]{BIN75}
$$
\mathbf{B}^{2-\frac{1}p,1-\frac{1}{2p}}_{p,p}(\R^{n-1}\times (0,T)) \rightarrow W^{2,1}_p (\R^n\times (0,T))
$$
(here the notation $\mathbf{B}$ of the Besov spaces corresponds to the book \cite{BIN75}).

Note that multiplying by a suitable cut-off function, 
the following properties can be ensured:
\begin{enumerate}
\item the extended function is equal to $0$ for $|x_n|>R/2$;
\item if the initial function is equal to $0$ for $|x'|>R/2$, then the extended one is equal to $0$ for $|x'|>3R/4$.
\end{enumerate}
Moreover, the norm of the extension operator is bounded in terms of ${n}$, $q$, $p$, and $R$ but not on $T$. 

\medskip

\textit{Step 2.} Since $\partial\Omega\in  \C^{1,1}$, for each $x^0 \in \partial\Omega$ there exists a Cartesian coordinate system with origin at $x^0$ such that the intersection of $\partial\Omega$ with the neighborhood 
$$
U_{\mathfrak{R}_0}=\big\{(x',x_n)\colon\ |x'|<{\mathfrak{R}_0}, |x_n|<{\mathfrak{R}_0}\big\}
$$
is the graph $x_n=f(x')$ of a function $f\in \C^{1,1}$.

The change of variables $y'=x'$, $y_n=x_n-f(x')$ then maps $\partial\Omega \cap U_{\mathfrak{R}_0}$ into the ball of radius ${\mathfrak{R}_0}$ lying on the hyperplane $\{y_n=0\}$.

This coordinate transformation induces the ``transplantation'' operator $u(x,t) \rightarrow u(y;t)$ acting from $W^{2,1}_q (\Gamma_T \cap (U_{\mathfrak{R}_0}\times(0,T))$ to $W^{2,1}_q (B_{\mathfrak{R}_0}'\times(0,T))$ continuously.
\medskip

\textit{Step 3.} Using the results of Steps 1 and 2 above, one can construct local extension operators that map $W^{2,1}_q (\Gamma_T)$-functions with sufficiently small $x$-support  into $W^{2,1}_p(\R^n\times(0,T))$-functions vanishing for $|x_n|>{\mathfrak{R}_0}/2$, $|x'|>3{\mathfrak{R}_0}/4$. Finally, the desired operator $\Pi$ can be glued  from the local operators via the appropriate partition of unity.
\end{proof}

\subsection{Coercive estimates for Cauchy--Dirichlet problems with discontinuous coefficients}

\begin{lemma}\label{CD-problem}
Consider the linear parabolic operator
$$
\CL_0:= \partial_t-a^{ij}(x,t)D_iD_j 
$$
and suppose
\begin{align*}
& a^{ij}(x,t)=a^{ji}(x,t)\quad \text{for a.a.}\ (x,t)\in \mathcal{Q}_T,\\
& \nu|\xi|^2\leq a^{ij}\xi_i\xi_j\leq \nu^{-1}|\xi|^2\quad \forall \xi\in\R^n,\quad \nu=\mathrm{const}>0,\\
& a^{ij}\in \textit{VMO}_x(\mathcal{Q}_T).
\end{align*}

Let $1<p<+\infty$, $\partial\Omega\in \C^{1,1}$ and suppose $u\in W^{2,1}_p(\mathcal{Q}_T)$ is a strong solution of the equation $\CL_0u=f\in L^p(\mathcal{Q}_T)$ such that $u=0$ on $\partial_P\mathcal{Q}_T$.

Then there exists a constant $C$ depending on $n$, $p$, $\nu$, $T$, the properties of $\partial\Omega$, and on the $\textit{VMO}_x$-moduli of the coefficients $a^{ij}$, such that
\begin{equation}\label{CD1}
\|u\|_{W^{2,1}_p(\mathcal{Q}_T)} \leq C\Big( \|f\|_{p,\mathcal{Q}_T}+
\|u\|_{p,\mathcal{Q}_T}\Big).
\end{equation}
\end{lemma}
\begin{proof}
According to the general $L^p$-estimate of \cite[Theorem~6]{DK11}, 
 there is a $\lambda_0\geq 1$ such that
$$
\|u\|_{W^{2,1}_p(\mathcal{Q}_T)} \leq C \|f-2\lambda_0u\|_{p,\mathcal{Q}_T},
$$
and this implies  
 \eqref{CD1}.
\end{proof}

\begin{remark} \label{rem1}
The same proof gives a coercive estimate for the solution of the Cauchy problem for the uniformly parabolic equation on $\Gamma_T$
$$
\CB_0u:= \partial_tu-\alpha^{ij}(x,t)d_id_j u=g
$$
with zero initial condition. Namely, if $u\in W^{2,1}_q(\Gamma_T)$, $1<q<\infty$, then
\begin{equation*}
\|u\|_{W^{2,1}_q(\Gamma_T)} \leq C\Big( \|g\|_{q,\Gamma_T}+
\|u\|_{q,\Gamma_T}\Big),
\end{equation*}
where $C$ depends on $n$, $q$, $\nu$, $T$, the properties of $\partial\Omega$, and on the $\textit{VMO}_x$-moduli of the coefficients $\alpha^{ij}$.
\end{remark}

\section{Solvability of the linear parabolic Venttsel problem}\label{sec3}

Assume $\partial\Omega \in \C^{1,1}$ and
let the exponents $q$ and $p$ be chosen such that (see Fig.~\ref{fig1})
\begin{equation} \label{pq-cond}
1<p \le q\,\frac{n+2}{n+1} <p^*:= \frac {(n+2)p}{(n+2-p)_+}\quad \text{and} \quad 
q >1.
\end{equation}

\begin{figure}[ht]
\centerline{\includegraphics[scale=.95]{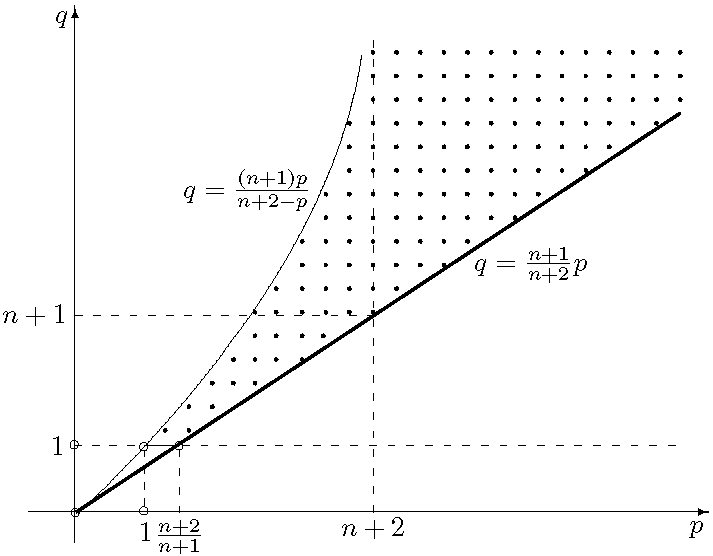}}
\vspace*{-.4cm}
\caption{The exponents $p$ and $q$}
\label{fig1}
\end{figure}

We introduce a linear parabolic operator $\CL$,
\begin{gather}
\label{4}
\CL := \partial_t-a^{ij}(x,t)D_iD_j+b^i(x,t)D_i+c(x,t),\\
\label{L1}
a^{ij}(x,t)=a^{ji}(x,t) \quad (x,t)\in \mathcal{Q}_T, \qquad a^{ij} \in \textit{VMO}_x\,(\mathcal{Q}_T), \tag{L1}\\
\label{L2}
\nu |\xi|^2 \le a^{ij}\xi_i\xi_j \le \nu^{-1} |\xi|^2 \quad \forall \xi\in \R^n, \quad \nu=\text{const}>0, \tag{L2}
\end{gather}
and a linear boundary operator $\CB$,
\begin{gather}
\label{5}
\CB := \partial_t-\alpha^{ij}(x,t)d_id_j+\beta^i(x,t)D_i+\gamma (x,t),\\
\label{B1}
\alpha^{ij}(x,t)=\alpha^{ji}(x,t) \quad (x,t)\in \Gamma_T, \qquad \alpha^{ij} \in \textit{VMO}_x\,(\Gamma_T), \tag{B1} \\
\label{B2}
\nu |\xi^*|^2 \le \alpha^{ij} \xi^*_i \xi^*_j \le \nu^{-1}|\xi^*|^2, \quad \forall\xi^*\in \R^n, \quad  \xi^* \perp \n (x). \tag{B2}
\end{gather}

Set $\mathbf{b}(x,t)=\big(b^1(x,t),\ldots,b^n(x,t)\big)$ and assume that the lower-order coefficients of the operator $\mathcal{L}$ satisfy the following integrability conditions
\begin{gather}
\label{L3}
\begin{aligned}
&|\mathbf{b}|\in L^{\max\{p,n+2\}}(\mathcal{Q}_T), && \text{if} \quad p \neq n+2,\\
&|\mathbf{b}| \left( \log{(1+|\mathbf{b}|)}\right)^{\frac {n+1}{n+2}} \in L^{n+2}(\mathcal{Q}_T), && \text{if} \quad p=n+2,
\end{aligned} 
\tag{L3}
\end{gather}
and
\begin{gather}
\label{L4}
\begin{aligned}
&c\in L^{\max\{p, \frac {n+2}2\}}(\mathcal{Q}_T), &&  \text{if} \quad p \neq \frac {n+2}2,\\
&c\left(\log{( 1+|c|)}\right)^{\frac n{n+2}} \in L^{\frac {n+2}2}(\mathcal{Q}_T), && \text{if} \quad p=\frac {n+2}2.
\end{aligned}
\tag{L4}
\end{gather}

The assumptions on the lower-order coefficients of the operator $\CB$ are as follows.  
We set $\boldsymbol{\beta}(x,t)=\big(\beta^1(x,t),\ldots,\beta^n(x,t)\big)$, define the normal and tangential components of the vector field $\boldsymbol{\beta} (x,t)$ as
$$
\beta_0(x,t):= \beta^i (x,t) \mathbf{n}_i(x); \quad \boldsymbol{\beta}^*(x,t):= \boldsymbol{\beta} (x,t) - \beta_0(x,t)  \mathbf{n} (x),
$$ 
and assume that
\begin{gather}
\label{B3}
\begin{aligned}
&|\boldsymbol{\beta}^*|\in L^{\max\{q,n+1\}}(\Gamma_T), &&  \text{if} \quad q \neq n+1,\\
&|\boldsymbol{\beta}^*| \left( \log{(1+|\boldsymbol{\beta}^*|)}\right)^{\frac n{n+1}} \in L^{n+1}(\Gamma_T), && \text{if} \quad q=n+1,
\end{aligned}
\tag{B3}
\end{gather}
\begin{gather}
\label{B5}
\begin{aligned}
&\beta_0\in L^q(\Gamma_T), && \text{if} \quad p >n+2,\\
&\beta_0\in L^{qp^*/(p^*-q\frac {n+2}{n+1})}(\Gamma_T), && \text{if} \quad p <n+2,\\
&\beta_0 \left( \log{(1+|\beta_0|)}\right)^{\frac {n+1}{n+2}} \in L^q(\Gamma_T), && \text{if} \quad p=n+2,
\end{aligned}
\tag{B4}
\end{gather}
\begin{gather}
\label{B4}
\begin{aligned}
&\gamma\in L^{\max\{q,\frac {n+1}2\}}(\Gamma_T),&& \text{if} \quad q \neq \frac {n+1}2,\\
&\gamma\left(\log{( 1+|\gamma|)}\right)^{\frac {n-1}{n+1}} \in L^{\frac {n+1}2}(\Gamma_T), && \text{if} \quad q=\frac {n+1}2.
\end{aligned}
\tag{B5}
\end{gather}

Our first result provides an \textit{a~priori} estimate for any strong solution to the linear parabolic Venttsel problem in terms of the data of the problem.
\begin{thm} \label{apriori-estimate}
Let the exponents $p$ and $q$ be chosen in accordance with \eqref{pq-cond},   $\partial\Omega \in \C^{1,1}$
and assume that conditions \eqref{L1}--\eqref{L4} and 
\eqref{B1}--\eqref{B4} are verified.  

Suppose that a function $u\in V_{p,q}(\mathcal{Q}_T)$  satisfies the equation
\begin{equation} \label{2.1-AN95}
\CL u(x,t)=f(x,t) \quad \text{a.e. in}\ \mathcal{Q}_T
\end{equation}
and the boundary condition
\begin{equation} \label{2.2-AN95}
\CB u(x,t)=g (x,t) \quad \text{a.e. on}\ \Gamma_T
\end{equation}
with $f \in L^p(\mathcal{Q}_T)$ and $g \in L^q(\Gamma_T)$. If $u\big|_{t=0}=0$ in $\overline{\Omega}$
then
\begin{equation} \label{2.3-AN95}
\|u\|_{V_{p,q}(\mathcal{Q}_T)} \le  C_1 \Big(\|f\|_{p,\mathcal{Q}_T}+\|g\|_{q, \Gamma_T}
\Big)
\end{equation}
with a constant $C_1$ depending  on $n$, $\nu$, $p$, $q$, $\mathrm{diam}\,\Omega$, $T$, the properties of $\partial\Omega$, on the $\textit{VMO}_x$-moduli  of the coefficients $a^{ij}(x,t)$ and  $\alpha^{ij}(x,t)$, and on the moduli of continuity of the functions $|\mathbf{b}|$, $c$, $|\boldsymbol{\beta}^*|$, $\beta_0$, and $\gamma$ in the corresponding functional spaces defined by conditions \eqref{L3}--\eqref{L4} and 
\eqref{B3}--\eqref{B4}, respectively.
\end{thm}

\begin{proof}
Similarly to \cite{ANPS21}, we apply the so-called \textit{Munchhausen trick} \cite{R86} and estimate the directional derivative $\partial_{\mathbf{n}} u$ in terms of itself.  
\medskip

\textit{Step 1.} 
Making use of 
$$
\beta^i(x,t)D_iu
={\beta^*}^i(x,t)d_iu+\beta_0(x,t)\partial_{\mathbf{n}}u,
$$ 
we rewrite the boundary condition \eqref{2.2-AN95} in the form
\begin{align} 
\label{2.14-AN95}
\CB_0 u &\ = g (x,t)-\beta_0 (x,t) 
\partial_{\mathbf{n}} u-{\beta^*}^i(x,t)d_iu-\gamma (x,t) u \\
\nonumber
&\ =: g_1 (x,t) \quad \text{a.e. on}\ \Gamma_T,
\end{align}
where the operator $\CB_0$ is defined in Remark~\ref{rem1}. According to this remark, we have
$$
\|u\|_{W^{2,1}_q (\Gamma_T)} \le N_1 \Big(\|g_1\|_{q, \Gamma_T}+\|u\|_{q,\Gamma_T}\Big).
$$ 

Employing \eqref{2.14-AN95} in the last inequality, it takes on the form 
\begin{align} 
\label{2.15-AN95}
\|u\|_{W^{2,1}_q (\Gamma_T)} \le N_1 \Big(\|g\|_{q, \Gamma_T} +\|u\|_{q, \Gamma_T}&+ \left\|\beta_0 
\partial_{\mathbf{n}} u\right\|_{q,\Gamma_T}\\
\nonumber
&+\|{\beta^*}^i d_iu\|_{q, \Gamma_T}+\|\gamma u\|_{q,\Gamma_T}\Big).
\end{align} 

To estimate  the last two terms in the right-hand of \eqref{2.15-AN95}, we claim that $u\mapsto{\beta^*}^id_iu$ and  $u\mapsto \gamma u$ are compact operators acting from $W^{2,1}_q (\Gamma_T)$ into $L^q(\Gamma_T)$. Indeed, the embedding theorems (see \cite[Theorem 10.4]{BIN75} for $q>n+1$, \cite[Theorem 10.2]{BIN75} for $q<n+1$, and \cite[Theorem 10.5]{BIN75} for $q=n+1$) and the assumptions \eqref{B3} show that
\begin{equation} \label{beta_num}
\|{\beta^*}^i d_iu\|_{q, \Gamma_T}\le C \|\boldsymbol{\beta^*}\|_{\mathcal{X}(\Gamma_T)}\|u\|_{W^{2,1}_q(\Gamma_T)},
\end{equation}
where $\mathcal{X}(\Gamma_T)$ stands for the corresponding functional space defined by relations \eqref{B3}.\footnote{This is the Lebesgue space if $q\ne n+1$ and the Orlicz space if $q=n+1$.}

Further, ${\beta^*}^i$ can be approximated in $\mathcal{X}(\Gamma_T)$ by bounded functions. Evidently, for bounded $\varphi^i$, the operator $u\mapsto \varphi^id_iu$ acting from $W^{2,1}_q (\Gamma_T)$ into $L^q(\Gamma_T)$ is compact. So, the estimate \eqref{beta_num} implies that the operator $u\mapsto{\beta^*}^id_iu$ can be norm-approximated by compact operators and therefore is compact itself. The operator $u\mapsto \gamma u$ is considered in the same way, and the claim follows.

It follows from compactness  that for arbitrary $\varepsilon>0$
\begin{equation}
\label{2.20-AN95}
\|{\beta^*}^id_iu\|_{q, \Gamma_T} +\|\gamma u\|_{q,\Gamma_T}
\le \varepsilon \|u\|_{W^{2,1}_q(\Gamma_T)} + N_2(\varepsilon)  \|u\|_{q, \Gamma_T},
\end{equation} 
where $N_2(\varepsilon)$ depends also on $n$, $q$, $T$, $\mathrm{diam}\,\Omega$, the properties of $\partial\Omega$ and on the moduli of continuity of the functions  $|\boldsymbol{\beta}^*|$ and $\gamma$ in the functional spaces defined by conditions  
\eqref{B3} and \eqref{B4}, respectively. 

Choosing $\varepsilon=\frac{1}{2N_1}$, we obtain from \eqref{2.15-AN95} and \eqref{2.20-AN95}
\begin{equation} \label{2.15-AN95-short}
\|u\|_{W^{2,1}_q(\Gamma_T)} \le N_3 \left( \|g\|_{q,\Gamma_T}+\|u\|_{q,\Gamma_T}+
\left\|\beta_0 
\partial_{\mathbf{n}} u\right\|_{q,\Gamma_T} \right),
\end{equation}
where $N_3=2N_1 N_2$. 

\medskip

\textit{Step 2.}
Consider in $\mathcal{Q}_T$ the function
\begin{equation} \label{2.16-AN95}
v(x,t)=u(x,t)-\widetilde{u}(x,t)
\end{equation}
with $\widetilde{u}=\Pi (u\big|_{\Gamma_T})$, where $\Pi$ is the extension operator constructed in Theorem~\ref{extension-theorem}. It is evident that $v$ solves the initial-boundary value problem
\begin{equation} \label{2.17-AN95}
\CL_0v=f_1(x,t) \quad \text{a.e. in}\ \mathcal{Q}_T,\qquad
v=0 \quad \text{on}\ \partial_P\mathcal{Q}_T,
\end{equation}
where $f_1(x,t):= f(x,t)-b^i(x,t)D_iu-c(x,t)u-\CL_0\widetilde{u}$, and  $\CL_0$ is defined in Lemma~\ref{CD-problem}.

By Lemma~\ref{CD-problem}, we have
$$
\|v\|_{W^{2,1}_p(\mathcal{Q}_T)} \le N_4 \Big(\|f_1\|_{p,\mathcal{Q}_T}+\|v\|_{p,\mathcal{Q}_T}\Big),
$$
where $N_4$ depends only on $n$, $p$, $\nu$, on the properties of $\partial\Omega$ and on the $\textit{VMO}_x$-moduli of the coefficients $a^{ij}$. In view of \eqref{2.16-AN95}, \eqref{6.1-AN95} and the definition of $f_1$, one can transform the latter inequality into
\begin{align} 
\label{2.18-AN95}
\|u\|_{W^{2,1}_p(\mathcal{Q}_T)} \le N_5 \Big(\|f\|_{p,\mathcal{Q}_T} +\|u\|_{p,\mathcal{Q}_T}&+\|u\|_{W^{2,1}_q(\Gamma_T)} \\
\nonumber
&+\|b^iD_iu\|_{p,\mathcal{Q}_T}+\|cu\|_{p,\mathcal{Q}_T}\Big)
\end{align}
and $N_5$ depends on the same quantities as $N_4$. Repeating the arguments from Step 1, we show that $u\mapsto b^iD_iu$ and  $u\mapsto c u$ are compact operators acting from $W^{2,1}_p (\mathcal{Q}_T)$ into $L^p(\mathcal{Q}_T)$. So, we estimate the last two terms on the right-hand side of \eqref{2.18-AN95} as did in deriving \eqref{2.15-AN95-short}, and thus
\begin{equation} \label{2.18-AN95-short}
\|u\|_{W^{2,1}_p(\mathcal{Q}_T)} \le N_6 \Big(\|f\|_{p,\mathcal{Q}_T} +\|u\|_{p,\mathcal{Q}_T}+\|u\|_{W^{2,1}_q(\Gamma_T)}\Big).
\end{equation}
Here $N_6$ is determined by the same parameters as $N_4$ and, in addition, it depends on the moduli of continuity of $|\mathbf{b}|$ and $c$ in the corresponding functional spaces given by \eqref{L3} and \eqref{L4}, respectively.

Combining  \eqref{2.15-AN95-short} with \eqref{2.18-AN95-short}, we get
\begin{align} \label{2.19-AN95}
\|u\|_{W^{2,1}_p(\mathcal{Q}_T)} \le N_7 \Big(\|f\|_{p,\mathcal{Q}_T} &+\|g\|_{q,\Gamma_T}\\
\nonumber
&+\|u\|_{p,\mathcal{Q}_T}+\|u\|_{q,\Gamma_T}+
\left\|\beta_0 \partial_{\mathbf{n}} u\right\|_{q,\Gamma_T}\Big),
\end{align}
where $N_7=N_6 \left(1+N_3\right)$. 

\medskip

\textit{Step 3.}  We are in a position now to estimate the term $\left\|\beta_0 \partial_{\mathbf{n}} u\right\|_{q,\partial\Omega}$. We use the trace embedding theorems (see \cite[Sec.~10.5-10.6]{BIN75}) and deduce, similarly to the previous steps, that $u\mapsto \beta_0 \partial_{\mathbf{n}}u$ is a compact operator acting from $W^{2,1}_p (\mathcal{Q}_T)$ into $L^q(\Gamma_T)$. This gives
\begin{equation} \label{beta_0}
\left\|\beta_0 \partial_{\mathbf{n}} u\right\|_{q,\Gamma_T} \le \frac{1}{2N_7} \|u\|_{W^{2,1}_p(\mathcal{Q}_T)} +N_8  \|u\|_{p, \mathcal{Q}_T},
\end{equation}
where $N_7$ is the constant from \eqref{2.19-AN95}, and $N_8$ is  determined by $n$, $p$, $T$, $\mathrm{diam}\, \Omega$, the properties of $\partial\Omega$ and on the moduli of continuity of $|\beta_0|$ in the functional spaces defined by conditions \eqref{B5}. \medskip

Substituting  \eqref{2.19-AN95} into \eqref{beta_0}, we finalize the Munchhausen trick and arrive at
$$
\left\|\beta_0 \partial_{\mathbf{n}} u\right\|_{q,\Gamma_T} \le (1+2N_8) \Big(\|f\|_{p,\mathcal{Q}_T} +\|g\|_{q,\Gamma_T}+\|u\|_{p,\mathcal{Q}_T}+\|u\|_{q,\Gamma_T}\Big).
$$
Inserting the last inequality into the right-hand sides of \eqref{2.15-AN95-short} and \eqref{2.19-AN95}  gives 
\begin{equation} \label{CD2-a}
\|u\|_{V_{p,q}(\mathcal{Q}_T)} \le  N_9 \Big(\|f\|_{p,\mathcal{Q}_T}+\|g\|_{q, \Gamma_T}
+\|u\|_{p,\mathcal{Q}_T}+\|u\|_{q,\Gamma_T} 
\Big).
\end{equation}
Finally, for a.a. $(x,t)\in \mathcal{Q}_T$ we have
$$
u(x,t)=\int\limits_0^t \partial_{\tau} u(x,\tau)\,d\tau
$$
whence
$$
\|u\|_{p,\mathcal{Q}_T} \leq T \|u\|_{W^{2,1}_p(\mathcal{Q}_T)}; \qquad
\|u\|_{q,\Gamma_T} \leq T \|u\|_{W^{2,1}_q(\Gamma_T)}.
$$
If $T>0$ is small enough, the norms $\|u\|_{p,\mathcal{Q}_T}$ and $\|u\|_{q, \Gamma_T}$ in \eqref{CD2-a} could be absorbed into $\|u\|_{V_{p,q}(\mathcal{Q}_T)}$ and this yields the desired bound \eqref{CD1}. Otherwise, we decompose $\mathcal{Q}_T$ as a finite union of cylinders $\Omega\times(t_i,t_{i+1})$ each of small enough height, and apply successively the bound just obtained to each of them. This completes the proof.
\end{proof}

Standard arguments, based on the parameter continuation and the coercive estimate \eqref{2.3-AN95}, lead to the following existence theorem.

\begin{thm}\label{existence}
Under the hypotheses of Theorem $\ref{apriori-estimate}$, 
the initial-boundary value problem \eqref{2.1-AN95}--\eqref{2.2-AN95} with zero initial condition admits a unique solution $u\in V_{p,q}(\mathcal{Q}_T)$ for arbitrary $f \in L^p(\mathcal{Q}_T)$ and $g \in L^q(\Gamma_T)$.
\end{thm}

\begin{remark}\label{rem-Krylov}
For $p<n+2$ the assumption \eqref{L3} can be relaxed in terms of Morrey spaces using \cite[Theorem 4.1]{Kr23} as follows:
$$
\lim\limits_{\rho\to0}\rho^{1-\frac{n+2}{p_1}}\!\sup\limits_{(x,t)\in\mathcal{Q}_T}\|b^i\|_{p_1,Q_\rho(x,t)\cap \mathcal{Q}_T}=0 \quad \mbox{for some} \quad p_1\in (p,n+2).
$$
In a similar way one can relax assumptions \eqref{L4} (for $p<\frac{n+2}2$), \eqref{B3} (for $q<n+1$), and \eqref{B4} (for $q<\frac{n+1}2$).

Notice also that the conditions $a^{ij} \in \textit{VMO}_x\,(\mathcal{Q}_T)$ and $\alpha^{ij} \in \textit{VMO}_x\,(\Gamma_T)$ can be weakened to the assumption that the quantities $(a^{ij})^{\#(x)}_R$ and $(\alpha^{ij})^{\#(x)}_R$ are sufficiently small for small $R$, cf. \cite{DK11}.
\end{remark}

Finally, we recall the global maximum estimate for solutions of linear parabolic Venttsel problems.\footnote{In fact, the assumptions on $\partial\Omega$ in the papers \cite{AN95} and \cite{AN95a} (cf. Proposition \ref{Holder-est}) are weaker than $\partial\Omega \in \mathcal{C}^{1,1}.$ \label{foot}}

\begin{prop}[Theorem~2.1 in \cite{AN95}]
\label{global-max-principle}
Let $\partial\Omega \in \mathcal{C}^{1,1}$. Suppose that the function $u \in  V_{n+1,n}(\mathcal{Q}_T)$
is a solution of \eqref{2.1-AN95}--\eqref{2.2-AN95}. 

Assume also that $\beta_0\ge 0$ a.e. on $\Gamma_T$ and 
$$
|\mathbf{b}|, f_+, c_- \in L^{n+1}(\mathcal{Q}_T); \qquad 
|\boldsymbol{\beta}|, g_+, \gamma_- \in L^n(\Gamma_T).
$$

Then the  estimate 
$$
\sup\limits_{\mathcal{Q}_T}u \le C_2 \Big(
\|f_+\|_{n+1,\mathcal{Q}_T}+\|g_+\|_{n, \Gamma_T}+\sup\limits_{\Omega}u(\cdot,0)
\Big)
$$
holds with $C_2$ depending  only on $n$, $\nu$, $T$, $\text{diam}\,\Omega$, 
the properties of $\partial\Omega$,  and on the moduli of continuity of the functions $b^i$, $c_-$ in $L^{n+1}(\mathcal{Q}_T)$ and the functions $\beta^i$, $\gamma_-$ in $L^n(\Gamma_T)$.     
\end{prop}

\section{The quasilinear parabolic Venttsel problem} \label{sec4}

We aim now to the study of the quasilinear parabolic equation
\begin{equation} \label{1.1-AN95}
\partial_tu-a^{ij}(x,t,u)D_iD_ju+a(x,t,u,Du)=0 \quad \text{a.e. in}\ \mathcal{Q}_T
\end{equation}
coupled with the quasilinear parabolic Venttsel boundary condition
\begin{equation} \label{1.2-AN95}
\partial_tu-\alpha^{ij}(x,t,u)d_id_iu+\alpha (x,t,u,Du)=0 \quad \text{a.e. on}\ \Gamma_T
 \end{equation}
and the initial condition 
\begin{equation} \label{1.3-AN95}
u\big|_{t=0}=0 \quad \text{in}\  \overline{\Omega}.
\end{equation}
As in the previous Section, we assume $\partial\Omega\in\mathcal{C}^{1,1}$.

Notice that \eqref{1.2-AN95} is not an autonomous equation on $\Gamma_T$ because it involves not only tangential derivatives but also the normal component of the gradient $Du$.

We suppose that $a^{ij}(x,t,z)$, $a(x,t,z,p)$, $\alpha^{ij}(x,t,z)$ and $\alpha(x,t,z,p)$ are Carath\'{e}odory functions, i.e., these are measurable with respect to $(x,t)$ for all $(z,p) \in \mathbb{R}\times\mathbb{R}^n$ and continuous with respect to $z$ and $p$ for almost all $(x,t)\in \mathcal{Q}_T$ (respectively, for almost all $(x,t)\in \Gamma_T$).

The equation \eqref{1.1-AN95} is assumed to be uniformly parabolic, that is, for almost all $ (x,t)\in \mathcal{Q}_T$ and for all $z\in \mathbb{R}$ we have
\begin{equation} 
\label{A1}
\begin{gathered} 
\nu |\xi|^2 \le a^{ij}(x,t,z)\xi_i\xi_j \le \nu^{-1} |\xi|^2 \quad \forall \xi\in \R^n, \ \nu=\text{const}>0, \\ a^{ij}(x,t,z)=a^{ji}(x,t,z).
\tag{A1}
\end{gathered}
\end{equation}

Regarding the regularity conditions of the coefficients $a^{ij}$, we suppose that
\begin{equation} 
\label{A2}
\begin{gathered}
a^{ij}(\cdot,z)\in \textit{VMO}_x\   \textit{locally uniformly in}\ z, \textit{that is}, \\ 
\lim_{r\to 0} \sup_{z\in [-M,M]} (a^{ij}(\cdot,z))^{\#(x)}_r=0,
\tag{A2}
\end{gathered}
\end{equation}
where $(a^{ij}(\cdot,z))^{\#(x)}_r$ is the $\textit{VMO}_x$-modulus of $a^{ij}(\cdot,z)$ defined by \eqref{VMO-mod} with $Q_{\rho}(x,t)$ replaced by $Q_{\rho}(x,t) \cap \mathcal{Q}_T$. Moreover, we need $a^{ij}$ to be   
locally uniformly continuous  with respect to $z$, uniformly in $(x,t)$. Namely, for all $z, \widehat{z}\in [-M,M]$  
\begin{equation} 
\label{A3}
\begin{gathered}
|a^{ij}(x,t,z)-a^{ij}(x,t,\widehat{z})|\le \tau_M(|z-\widehat{z}|)\quad \textit{a.e. in}\ \mathcal{Q}_T,
\\ 
\textit{with a non-decreasing function}\ 
\tau_M(\zeta),\ \lim_{\zeta \to 0}\tau_M (\zeta)=0.
\tag{A3}
\end{gathered}
\end{equation}

The function $a(x,t,z,p)$ is assumed to grow quadratically  with respect to the gradient, i.e., for almost all $(x,t)\in \mathcal{Q}_T$ and for all $(z,p)\in \mathbb{R}\times\mathbb{R}^n$
\begin{equation} 
\label{A4}
|a(x,t,z,p)|\le \eta(|z|)\Big(\mu |p|^2+b(x,t)|p|+\Phi(x,t)\Big)
\tag{A4}
\end{equation}
where $\mu\ge0$ is a constant, $\eta\in\mathcal{C}(\mathbb R_+)$ is a non-decreasing function, and 
\begin{equation} \label{A5}
 b \left(\log{(1+|b|)}\right)^{\frac{n+1}{n+2}} \in L^{n+2}(\mathcal{Q}_T),\quad \Phi \in L^{n+2}(\mathcal{Q}_T).
\tag{A5}
\end{equation}

Further on, we assume that the boundary condition 	\eqref{1.2-AN95} is a uniformly parabolic Venttsel condition in the sense that  
for almost all $(x,t)\in \Gamma_T$ and for all $(z,p) \in \mathbb{R}\times \mathbb{R}^n$ we have:
\begin{equation} \label{V}
\begin{gathered}
\textit{the function}\ \alpha(x,t,z,p)\ \textit{is weakly differentiable with respect to}\ p_i,\ \textit{and}\\
0 \le \alpha_{p_i}(x,t,z,p)\mathbf{n}_i(x) \le \eta(|z|)\beta_0(x,t),
\tag{V}
\end{gathered}
\end{equation}
with $\eta$ as above and
\begin{equation} \label{V0}
\beta_0 \left(\log{(1+|\beta_0|)}\right)^{\frac{n+1}{n+2}} \in L^{n+1}(\Gamma_T),
\tag{V0}
\end{equation}
and
\begin{equation} \label{V1}
\begin{gathered}
\nu |\xi^*|^2 \le \alpha^{ij}(x,t,z)\xi^*_i\xi^*_j \le \nu^{-1} |\xi^*|^2 \quad \forall \xi^*\in \R^n, \quad \xi^* \perp \mathbf{n}(x),\\ \nu=\text{const}>0, \quad \alpha^{ij}(x,t,z)=\alpha^{ji}(x,t,z).
\tag{V1}
\end{gathered}
\end{equation}
In addition, we impose regularity  conditions on the coefficients $\alpha^{ij}$ similar to these required for $a^{ij}$. Precisely,
\begin{equation} \label{V2}
\begin{gathered}
\alpha^{ij}(\cdot ,z)\in \textit{VMO}_x\   \textit{locally uniformly in}\ z, \ \textit{that is},\\
\lim_{r\to 0} \sup_{z\in [-M,M]} (\alpha^{ij}(\cdot,z))^{\#(x)}_r=0,
 \tag{V2}
\end{gathered}
\end{equation}
where $(\alpha^{ij}(\cdot,z))^{\#(x)}_r$ is the $\textit{VMO}_x$-modulus of $\alpha^{ij}(\cdot,z)$ defined by \eqref{VMO-mod} with $Q_{\rho}(x,t)\cap \Gamma_T$ in the place of $Q_{\rho}(x,t)$,
and for all $z, \widehat{z}\in [-M,M]$,
\begin{equation} \label{V3}
\begin{gathered}
|\alpha^{ij}(x,t,z)-\alpha^{ij}(x,t,\widehat{z})|\le {\tau}_M(|z-\widehat{z}|)\quad \textit{a.e. on}\ \Gamma_T,
\\ \textit{with}\ 
{\tau}_M(\zeta)\
\textit{as in}\ \eqref{A3}.
\tag{V3}
\end{gathered}
\end{equation}

The term $\alpha(x,t,z,p)$ of \eqref{1.2-AN95} is required to support quadratic growth with respect to the tangential gradient, i.e., for almost all $(x,t)\in \Gamma_T$ and for all $z\in \mathbb{R}$, $p^*\in\mathbb{R}^n$, $p^*\perp \mathbf{n}(x)$
\begin{equation} \label{V4}
|\alpha(x,t,z,p^*)|\le \eta(|z|)\Big(\mu |p^*|^2+\beta (x,t)|p^*|+\Theta (x,t)\Big)
\tag{V4}
\end{equation}
with $\mu$ and $\eta$ as in \eqref{A4}, and 
\begin{equation} \label{V5}
\beta \left(\log{(1+|\beta|)}\right)^{\frac{n}{n+1}} \in L^{n+1}(\Gamma_T),\quad 
\Theta \in L^{n+1}(\Gamma_T).
\tag{V5}
\end{equation}

\begin{remark} \label{same-way}
Arguing in the same way as in \cite[Lemma~2.6.2]{MPS00} it can be shown that conditions \eqref{A2}--\eqref{A3} and \eqref{V2}--\eqref{V3} provide for $u\in \C(\overline{\mathcal{Q}}_T)$ the inclusions 
$$
a^{ij}(x,t,u(x,t)) \in \textit{VMO}_x \cap L^{\infty}(\mathcal{Q}_T),\quad \alpha^{ij}(x,t,u(x,t)) \in \textit{VMO}_x \cap L^{\infty}(\Gamma_T)
$$ 
with $\textit{VMO}_x$-moduli bounded in terms of the continuity modulus of $u(x,t)$ and $\sup_{\mathcal{Q}_T}|u|$.  
\end{remark}

The strong solvability of the problem \eqref{1.1-AN95}--\eqref{1.3-AN95}  in the space $V_{n+2,n+1}(\mathcal{Q}_T)$ will be proved by the aid of the Leray--Schauder fixed point theorem. To apply it, we have to derive \textit{a~priori} estimates in a suitable functional space for any solution to a family of quasilinear parabolic Venttsel problems. Following the classical approach of O.A.~Ladyzhenskaya and N.N.~Ural'tseva~\cite{LU86}, we obtain these estimates assuming  we already dispose of a bound for the supremum norm $\sup_{\mathcal{Q}_T}|u|$ of a solution $u \in  V_{n+2,n+1}(\mathcal{Q}_T)$.

\medskip

We recall, first of all, the \textit{a~priori} estimate for the H{\"o}lder norm of a strong solution.

\begin{prop}[Theorem 1 in \cite{AN95a}] \label{Holder-est}
Let $\partial\Omega \in \mathcal{C}^{1,1}$. Suppose that the function $u \in  V_{n+1,n}(\mathcal{Q}_T)$
is a solution of \eqref{1.1-AN95}--\eqref{1.3-AN95}. 

Assume also that conditions \eqref{A1}, \eqref{A4}, \eqref{V}, \eqref{V1} and \eqref{V4} are satisfied with
$$
b\in L^{n+2}(\mathcal{Q}_T);\quad \Phi\in L^{n+1}(\mathcal{Q}_T);\quad \beta_0,\beta\in L^{n+1}(\Gamma_T);\quad \Theta\in L^{n}(\Gamma_T).
$$

Then there exists a constant $\lambda >0$ depending only on $n$, $\nu$ and the properties of $\partial\Omega$, such that 
$$
\|u\|_{\mathcal{C}^{0,\lambda} (\overline{\mathcal{Q}_T})}\le M_{\lambda},
$$
where $M_{\lambda}$ depends only on $n$, $\nu$, $\mu$,  the properties of $\partial\Omega$, $M_0=\sup_{\mathcal{Q}_T}|u|$, ${\eta}(M_0)$, $\|\Phi\|_{n+1,\mathcal{Q}_T}$, $\|\Theta\|_{n,\Gamma_T}$ and on the moduli of continuity of the functions $b$, $\beta_0$ and $\beta$ in the corresponding Lebesgue spaces. 
\end{prop}

The key \textit{a~priori} estimate of our approach is the gradient one.

\begin{theorem}\label{Du-est}
Let $\partial\Omega \in \C^{1,1}$ and assume  that conditions \eqref{A1}--\eqref{A5}, \eqref{V}, \eqref{V0}--\eqref{V5} are satisfied.

Then any solution
 $u \in  V_{n+2,n+1}(\mathcal{Q}_T)$ of the problem
\eqref{1.1-AN95}--\eqref{1.3-AN95} fulfills
the estimate
\begin{equation}
\label{grad-est}
\|Du\|_{2(n+2),\mathcal{Q}_T}+\|du\|_{2(n+1),\Gamma_T} \le M_1
\end{equation}
with a constant $M_1$ depending on:
\begin{itemize}
    \item $n$, $\nu$, $T$, $\mathrm{diam}\,\Omega$ and the properties of $\partial\Omega$;
    \item $M_0=\sup_{\mathcal{Q}_T}|u|$ and ${\eta}(M_0)$;
    \item the norms $\|\Phi\|_{n+2,\mathcal{Q}_T}$ and $\|\Theta\|_{n+1,\Gamma_T}$;
    \item the constant $\mu$ and the moduli of continuity of the functions $b$, $\beta_0$ and $\beta$ in the Orlicz spaces defined by conditions \eqref{A5}, \eqref{V0} and \eqref{V5}, respectively;
    \item the $\textit{VMO}_x$-moduli w.r.t. the independent variables and on the moduli of continuity w.r.t. $z$  of the Carath\'eodory functions $a^{ij}(x,t,z)$ and $\alpha^{ij}(x,t,z)$, see conditions \eqref{A2}--\eqref{A3}, \eqref{V2}--\eqref{V3}.
\end{itemize}
\end{theorem}

\begin{remark}
We point out that for any $u \in  V_{n+2,n+1}(\mathcal{Q}_T)$ the norms in~\eqref{grad-est} are finite by the embedding theorem, and Theorem~\ref{Du-est} deals only with {\rm a priori} estimate.
\end{remark}

\begin{proof}
To derive the estimate \eqref{grad-est} we will apply a homotopy argument due to Amann and Crandall \cite{AC78} that has already been  used in \cite{ANPS21} in the study of quasilinear elliptic Venttsel problems.

For the sake of brevity, we define hereafter
\begin{align}
\tilde{a}^{ij}(x,t):=&\ {a}^{ij}(x,t,u(x,t)),
\label{tilde-aij}\\ 
 \tilde{a}(x,t):=&\ \dfrac{a(x,t,u(x,t),Du(x,t))}{\mu|Du(x,t)|^2+b(x,t)|Du(x,t)|+\Phi(x,t)}, \label{tilde-a}\\
\tilde{b}^i(x,t):=&\ 
b(x,t)\tilde{a}(x,t)\, \dfrac{D_iu(x,t)}{|Du(x,t)|}
\label{tilde-b}
\end{align}
(recall that we set $\frac{0}{0}=0$, if such an uncertainty occurs)
and note that the assumption \eqref{A1} implies immediately the uniform ellipticity condition \eqref{L2} for $\tilde{a}^{ij}$.

Further, we make use of the hypotheses \eqref{A2}, \eqref{A3} and employ Remark~\ref{same-way} in order to get that the $\textit{VMO}_x$-moduli of $\tilde{a}^{ij}(x,t)$ are controlled in terms of $M_0=\sup_{\mathcal{Q}_T}|u|$ and the modulus of continuity of $u$. Further, Proposition~\ref{Holder-est} provides estimates 
 the continuity modulus of $u$ in terms of $M_0$ and therefore
 $\tilde{a}^{ij}$ satisfy \eqref{L1}. 
Moreover, $|\tilde a|\le {\eta}(M_0)$ in view of \eqref{A4}
while $\tilde b^i$ verify \eqref{L3} as consequence of \eqref{A5}.

This way, the equation \eqref{1.1-AN95} can be rewritten as
\begin{multline}
\label{VP1}
\partial_tu-\tilde{a}^{ij}(x,t)D_iD_ju+\tilde{b}^i(x,t)D_iu
\\
+\mu\tilde{a}(x,t)|Du|^2+\Phi(x,t) u
=\tilde{f}(x,t)\quad \text{a.e. in}\  \mathcal{Q}_T
\end{multline}
with
$$
\tilde{f}(x,t):= \Phi(x,t)\big(u(x,t)-\tilde{a}(x,t)\big)\in L^{n+2}(\mathcal{Q}_T).
$$

Similarly, we define
\begin{align}
\label{tilde-alpij}
\tilde{\alpha}^{ij}(x,t):=&\ {\alpha}^{ij}(x,t,u(x,t)), \\ 
\label{tilde-alpha}
 \tilde{\alpha}(x,t):=&\ \dfrac{\alpha(x,t,u(x,t),du(x,t))}{\mu|du(x,t)|^2+\beta(x,t)|du(x,t)|+\Theta(x,t)},  
 \end{align}
\begin{align}
\label{t-beta-prime} 
\tilde{\beta}^{*i}(x,t):=&\ 
\beta(x,t)\tilde{\alpha}(x,t)\, \dfrac{d_iu(x,t)}{|du(x,t)|}, \\
\label{beta0u}
\tilde\beta_0(x,t):= &\ \int\limits_0^1 \alpha_{p_i}\big(x,t,u(x,t),du(x,t)+s{\mathbf{n}}\partial_{\mathbf{n}}u(x,t)\big)\mathbf{n}_i(x)\,ds 
\end{align}
and note that  the assumptions \eqref{V1}--\eqref{V3} and Remark~\ref{same-way} imply the conditions \eqref{B1}--\eqref{B2} for $\tilde{\alpha}^{ij}$. Further on, $|\tilde \alpha|\le {\eta}(M_0)$ by \eqref{V4}, $\tilde\beta^{*i}$ satisfy \eqref{B3} because of \eqref{V5}, while \eqref{V} implies
$0\le \tilde\beta_0(x,t)\le {\eta}(M_0)\beta_0(x,t)$
for a.a. $(x,t)\in\Gamma_T$. 
Moreover,
$$
\alpha(x,t,u,Du)=\alpha(x,t,u,du)+\tilde\beta_0(x,t)\partial_{\mathbf{n}}u
$$
and the boundary equation \eqref{1.2-AN95} takes on the form
\begin{multline}
\label{VP2}
\partial_tu-\tilde{\alpha}^{ij}(x,t)d_id_ju
+ \tilde\beta^{*i}(x,t)d_iu+\tilde\beta_0(x,t)\partial_{\mathbf{n}}u\\
+ \mu\tilde{\alpha}(x,t)|du(x,t)|^2 +\Theta(x,t) u=\tilde{g}(x,t)
\quad \text{a.e. on}\ \Gamma_T,
\end{multline}
where $\tilde{g}(x,t):= \Theta(x,t) \big(u(x,t)-\tilde\alpha(x,t)\big) \in L^{n+1}(\Gamma_T)$. 
\smallskip

Consider now 
the one-parameter family of parabolic Venttsel problems
\begin{align}
\label{VP1-delta}
\partial_tv_\delta-\tilde{a}^{ij} D_{i}D_jv_\delta + \tilde{b}^i D_iv_\delta & + \mu\tilde{a}|Dv_\delta|^2+\Phi v_\delta=\delta\tilde{f} \quad \text{a.e. in}\ \mathcal{Q}_T,\\
\nonumber
\partial_tv_\delta-\tilde{\alpha}^{ij}d_{i}d_jv_\delta + \tilde\beta^{*i} d_iv_\delta  & +\tilde\beta_0\partial_{\mathbf{n}}v_\delta
 \\
\label{VP2-delta}
& + \mu\tilde{\alpha}|dv_\delta|^2+\Theta v_\delta
=\delta\tilde{g} \quad \text{a.e. on}\ \Gamma_T,\\
\label{VP3-delta}
&\qquad \qquad \ \quad  v_{\delta}\big|_{t=0}=0 \quad \ \text{in}\ \overline{\Omega}.
\end{align}

Following the strategy of \cite{AC78}, we will prove unique solvability of  \eqref{VP1-delta}--\eqref{VP3-delta}  in $V_{n+2,n+1}(\mathcal{Q}_T)$
 for each $\delta\in[0,1]$ and will estimate the gradient of $v_{\delta_2}$ in terms of  the gradient of
$v_{\delta_1}$ for small enough $\delta_2-\delta_1>0$. Then we will easily have $v_0\equiv0$, while for $\delta=1$ the coincidence of \eqref{VP1-delta} with \eqref{VP1} and \eqref{VP2-delta} with \eqref{VP2}   would give $v_1=u$. So, finite iteration in $\delta$ will  give the desired bound \eqref{grad-est}. 

To realize that plan, we need two lemmata.
\begin{lemma}
\label{lemma-L-infty}
Suppose that $v_{\delta_1},v_{\delta_2}\in V_{n+2,n+1}(\mathcal{Q}_T)$ solve \eqref{VP1-delta}--\eqref{VP3-delta} with $\delta_1\le\delta_2$.

 Then
\begin{equation}
\label{36}
\|v_{\delta_1}-v_{\delta_2}\|_{\infty,\Omega}
\le (\delta_2-\delta_1)\big(M_0+{\eta}(M_0)\big)
\end{equation}
(recall that $M_0=\sup_{\mathcal{Q}_T}|u|$).
\end{lemma}

\begin{proof}
Setting $w=v_{\delta_1}-v_{\delta_2}$ we obtain the following \textit{linear} parabolic Venttsel problem 
\begin{align*}
\tilde{\mathcal{L}}w:= \partial_tw-\tilde{a}^{ij}D_iD_jw + (\tilde{a}^i+\tilde{b}^i)D_iw & +\Phi w= (\delta_1-\delta_2)\tilde{f} \quad \text{a.e. in}\ \mathcal{Q}_T,\\
\tilde{\mathcal{B}}w:= \partial_tw-\tilde{\alpha}^{ij}d_id_jw +
(\tilde{\alpha}^i+\tilde\beta^{*i})d_iw & +\tilde\beta_0\partial_{\mathbf{n}}w\\
& +\Theta w =(\delta_1-\delta_2)\tilde{g} \quad \text{a.e. on}\ \Gamma_T
\end{align*}
with homogeneous initial data. Here
\begin{align*}
\tilde{a}^i(x,t):=&\ \mu\tilde{a}(x,t)\big(D_iv_{\delta_1}(x,t)+D_iv_{\delta_2}(x,t)\big),\\
\tilde{\alpha}^i(x,t):=&\ \mu\tilde{\alpha}(x,t)\big(d_iv_{\delta_1}(x,t)+d_iv_{\delta_2}(x,t)\big).
\end{align*}
We recall that $|\tilde{a}|\le{\eta}(M_0)$ and $v_{\delta_1},v_{\delta_2}\in W^{2,1}_{n+2}(\mathcal{Q}_T)$, thus $\tilde{a}^i\in L^{n+2}(\mathcal{Q}_T)$. Similarly, $\tilde{\alpha}^i\in L^{n+1}(\Gamma_T)$.

Using $\tilde{f}(x,t)\ge -\Phi(x,t)\big(M_0+{\eta}(M_0)\big)$ we get
$$
\tilde{\mathcal{L}}w\le 
\tilde{\mathcal{L}}\big((\delta_2-\delta_1)(M_0+{\eta}(M_0))\big)
\qquad \text{a.e. in}\ \mathcal{Q}_T
$$
and similarly
$$
\tilde{\mathcal{B}}w\le 
\tilde{\mathcal{B}}\big((\delta_2-\delta_1)(M_0+{\eta}(M_0))\big)\qquad \text{a.e. on}\ \Gamma_T.
$$
It follows from Proposition~\ref{global-max-principle} that
$$
w(x,t)\le (\delta_2-\delta_1)\big(M_0+{\eta}(M_0)\big)\qquad \text{in}\ \mathcal{Q}_T.
$$
In the same manner the lower estimate $w(x,t)\ge -(\delta_2-\delta_1)\big(M_0+{\eta}(M_0)\big)$ follows and this gives the claim \eqref{36}.
\end{proof}

It is worth noting that setting $\delta_1=\delta_2$ in \eqref{36}, we get immediately $v_{\delta_1}\equiv v_{\delta_2}$ and thus uniqueness of solutions to \eqref{VP1-delta}--\eqref{VP3-delta}. 
Precisely,
\begin{cor}
\label{cor-uniqueness}
The problem \eqref{VP1-delta}--\eqref{VP3-delta} cannot have more than one solution in $V_{n+2,n+1}(\mathcal{Q}_T)$ for any $\delta\in[0,1]$.
\end{cor}

\begin{lemma}\label{lemma-grad}
Under the hypotheses of Lemma $\ref{lemma-L-infty}$, there is a $\varkappa>0$ such that the inequality $\delta_2-\delta_1\le\varkappa$ implies
\begin{multline}
\label{39}
\|Dv_{\delta_2}-Dv_{\delta_1}\|_{2(n+2),\mathcal{Q}_T} 
+ \|dv_{\delta_2}-dv_{\delta_1}\|_{2(n+1),\Gamma_T}\\ \le\ C_3 (\delta_2-\delta_1) \Big(1+\|Dv_{\delta_1}\|_{2(n+2),\mathcal{Q}_T} + \|dv_{\delta_1}\|_{2(n+1),\Gamma_T}\Big).
\end{multline}
The constants $\varkappa$ and $C_3$ depend on the same quantities as $M_1$ in the statement of Theorem~$\ref{Du-est}$.
\end{lemma}
\begin{proof} We rewrite the problem for
$w=v_{\delta_1}-v_{\delta_2}$ as follows:
\begin{align}
\label{hat-L}
\widehat{\mathcal{L}}w:=&\
\partial_tw-\tilde{a}^{ij}D_iD_jw + \tilde{b}^iD_iw+\Phi w= \widehat{f} &\text{a.e. in}\ \mathcal{Q}_T,\\
\label{hat-B}
\widehat{\mathcal{B}}w:=&\
 \partial_tw-\tilde{\alpha}^{ij}d_id_jw +
\tilde\beta^{*i}d_iw +\tilde\beta_0\partial_{\mathbf{n}}w+\Theta w =\widehat{g} &\text{a.e. on}\ \Gamma_T,
\end{align}
with homogeneous initial data. Here
\begin{align*}
\widehat f= &\ (\delta_1-\delta_2)\tilde{f}-\mu
\tilde{a}\big(|Dv_{\delta_1}|^2-|Dv_{\delta_2}|^2\big)\in L^{n+2}(\mathcal{Q}_T),\\
\widehat g = &\ 
(\delta_1-\delta_2)\tilde{g}-\mu
\tilde{\alpha}\big(|dv_{\delta_1}|^2-|dv_{\delta_2}|^2\big)\in L^{n+1}(\Gamma_T).
\end{align*}
Theorem~\ref{apriori-estimate} yields
\begin{equation}
\label{37}
\|w\|_{V_{n+2,n+1}(\Omega)}\le N_9 \left( 
\big\|\widehat{f}\big\|_{n+2,\mathcal{Q}_T}+
\big\|\widehat{g}\big\|_{n+1,\Gamma_T}\right),
\end{equation}
where $N_9$ depends only on $n$, $\nu$, $\mathrm{diam}\,\Omega$, $T$, the properties of $\partial\Omega$, $\eta(M_0)$, the norms $\|\Phi\|_{n+2,\mathcal{Q}_T}$ and $\|\Theta\|_{n+1,\Gamma_T}$, the moduli of continuity of the functions $b$, $\beta_0$ and $\beta$ in the corresponding Orlicz spaces defined by conditions \eqref{A5}, \eqref{V0} and \eqref{V5}, respectively, and on the $\textit{VMO}_x$-moduli of the coefficients $\tilde{a}^{ij}(x,t)$  and $\tilde{\alpha}^{ij}(x,t)$. 

However, as explained before, the 
 $\textit{VMO}_x$-moduli of $\tilde{a}^{ij}(x,t)$  and $\tilde{\alpha}^{ij}(x,t)$ are controlled in terms of $M_0$ through Remark~\ref{same-way} and Proposition~\ref{Holder-est}. Thus, the constant $N_9$ in \eqref{37} depends only on data listed in the statement of Theorem~\ref{Du-est}.

Taking advantage of the bounds
\begin{align*}
\|\tilde{f}\|_{n+2,\mathcal{Q}_T}&\le  \|\Phi\|_{n+2,\mathcal{Q}_T}\big(M_0+{\eta}(M_0)\big),\\
\|\tilde{g}\|_{n+1,\Gamma_T}&\le  \|\Theta\|_{n+1,\Gamma_T}\big(M_0+{\eta}(M_0)\big),
\end{align*}
and of the evident inequalities
\begin{align*}
\left\| |Dv_{\delta_1}|^2-|Dv_{\delta_2}|^2\right\|_{n+2,\mathcal{Q}_T}\le &\ \|Dw\|^2_{2(n+2),\mathcal{Q}_T}\\
+&\ 2\|Dw\|_{2(n+2),\mathcal{Q}_T}\|Dv_{\delta_1}\|_{2(n+2),\mathcal{Q}_T},\\
\left\| |dv_{\delta_1}|^2-|dv_{\delta_2}|^2\right\|_{n+1,\Gamma_T}\le &\ \|dw\|^2_{2(n+1),\Gamma_T}+2\|dw\|_{2(n+1),\Gamma_T}\|dv_{\delta_1}\|_{2(n+1),\Gamma_T},
\end{align*}
we rewrite \eqref{37} as follows:
\begin{align}
\label{38}
&\|w\|_{V_{n+2,n+1}(\mathcal{Q}_T)}\le \ N_{10} \Big( (\delta_2-\delta_1)+
\|Dw\|^2_{2(n+2),\mathcal{Q}_T}+\|dw\|^2_{2(n+1),\Gamma_T}\\
\nonumber
 +&\ \|Dw\|_{2(n+2),\mathcal{Q}_T}\|Dv_{\delta_1}\|_{2(n+2),\mathcal{Q}_T} +\|dw\|_{2(n+1),\Gamma_T}\|dv_{\delta_1}\|_{2(n+1),\Gamma_T}\Big),
\end{align}
where $N_{10}$ depends on the same quantities as $N_9$. 

We infer now the anisotropic  Gagliardo--Nirenberg type inequality \cite[Theorem~4]{Sol72}
and the estimate \eqref{36} to get
\begin{align}
\label{Dw}
\|Dw\|^2_{2(n+2),\mathcal{Q}_T}\le&\ C(n,\Omega,T)\big(\|D^2w\|_{n+2,\mathcal{Q}_T}+\|w\|_{\infty,\mathcal{Q}_T}\big)\|w\|_{\infty,\mathcal{Q}_T}\\
\nonumber
\le &\ C(n,\Omega,T)(\delta_2-\delta_1)\big(M_0+{\eta}(M_0)\big)\|w\|_{V_{n+2,n+1}(\mathcal{Q}_T)},
\end{align}
and similarly
\begin{equation}\label{dw}
\|dw\|^2_{2(n+1),\Gamma_T}\le
C(n,\Omega,T)(\delta_2-\delta_1)\big(M_0+{\eta}(M_0)\big)\|w\|_{V_{n+2,n+1}(\mathcal{Q}_T)}.
\end{equation}
We substitute these inequalities into \eqref{38} and estimate the last two terms by the Cauchy inequality. This gives
\begin{align*}
\|w\|_{V_{n+2,n+1}(\mathcal{Q}_T)}\le &\ N_{11} \Big((\delta_2-\delta_1) +(\delta_2-\delta_1+\varkappa)\|w\|_{V_{n+2,n+1}(\mathcal{Q}_T)}\\
+&\ \frac{\delta_2-\delta_1}{\varkappa}
\big(\|Dv_{\delta_1}\|^2_{2(n+2),\mathcal{Q}_T}+\|dv_{\delta_1}\|^2_{2(n+1),\Gamma_T}\big)\Big)
\end{align*}
with arbitrary $\varkappa\in(0,1)$. Here
$N_{11}$ depends on the same quantities as $N_9$. 
Choosing $\varkappa=\frac 1{4N_{11}}$ we obtain \eqref{39} for $\delta_2-\delta_1\le\varkappa$, in view of \eqref{Dw} and \eqref{dw}.
\end{proof}

\medskip

Turning back to
the proof of Theorem \ref{Du-est}, we fix $\delta_1=0$ and $\delta_2=\varkappa$ in \eqref{39} and remember that $v_0\equiv0$ by Corollary~\ref{cor-uniqueness}. This gives the \textit{a~priori} estimate
\begin{equation}
\label{40}
\|Dv_{\varkappa}\|_{2(n+2),\mathcal{Q}_T}+\|dv_{\varkappa}\|_{2(n+1),\Gamma_T} \le C_3\varkappa.
\end{equation}

The solvability of \eqref{VP1-delta}--\eqref{VP3-delta} with $\delta=\varkappa$ is a consequence of the Leray--Schauder fixed point theorem. Indeed, we introduce the space
\begin{equation*}
{\mathcal V}_{p,q}(\mathcal{Q}_T)=\big\{v\in {\cal C}(\overline{\mathcal{Q}}_T)\ \colon \  Dv\in L^p(\mathcal{Q}_T),\ \ dv\in L^q(\Gamma_T) \big\}
\end{equation*}
equipped with the natural norm
$$
\|v\|_{{\mathcal V}_{p,q}(\mathcal{Q}_T)}=\|Dv\|_{p,\mathcal{Q}_T}+\|dv\|_{q,\Gamma_T}+\sup_{\mathcal{Q}_T}|v|,
$$
and define the nonlinear operator 
$$
\mathcal{F}: \ {\mathcal V}_{2(n+2),2(n+1)}(\mathcal{Q}_T) \mapsto V_{n+2,n+1}(\mathcal{Q}_T)
$$ 
which associates to any $w\in {\mathcal V}_{2(n+2),2(n+1)}(\mathcal{Q}_T)$ the unique solution $v=\mathcal{F}(w)$ of the \textit{linear} parabolic Venttsel problem
\begin{align*}
\widehat{\mathcal{L}}v=&\  \varkappa \tilde{f}(x,t)-
\tilde{a}(x,t)|Dw|^2 &\text{a.e.\ }& \text{in}\ \mathcal{Q}_T, \\
\widehat{\mathcal{B}}v =&\ \varkappa\tilde{g}(x,t)-
\tilde{\alpha}(x,t)|dw|^2 &\text{a.e.\ }& \text{on}\ \Gamma_T,\\
v\big|_{t=0}&=0 &&\text{in}\ \overline{\Omega}
\end{align*}
with operators $\widehat{\mathcal{L}}$ and $\widehat{\mathcal{B}}$ given by \eqref{hat-L} and \eqref{hat-B}, respectively.

The unique solvability of that problem follows from Theorem~\ref{existence} and Proposition~\ref{global-max-principle} 
due to assumptions of Theorem \ref{Du-est} and $w\in {\mathcal V}_{2(n+2),2(n+1)}(\mathcal{Q}_T)$. Therefore, the nonlinear operator $\mathcal{F}$ is well defined. Moreover, the problem \eqref{VP1-delta}--\eqref{VP3-delta} with $\delta=\varkappa$ is equivalent to the equation $u=\mathcal{F}(u)$. 

The estimate \eqref{2.3-AN95} yields the continuity of $\mathcal{F}$, while the compactness of the embedding $V_{n+2,n+1}(\mathcal{Q}_T) \hookrightarrow {\mathcal V}_{2(n+2),2(n+1)}(\mathcal{Q}_T)$ guarantees the compactness of $\mathcal{F}$ considered as a mapping from ${\mathcal V}_{2(n+2),2(n+1)}(\mathcal{Q}_T)$ into itself. Finally, any solution of the equation $v=\sigma\mathcal{F}(v)$, $0\le\sigma\le1$, that is,
\begin{align*}
\widehat{\mathcal{L}}v=&\  \sigma \big(\varkappa \tilde{f}(x,t)-
\tilde{a}(x,t)|Dv|^2\big) &\text{a.e.\ }& \text{in}\ \mathcal{Q}_T, \\
\widehat{\mathcal{B}}v =&\ \sigma \big(\varkappa\tilde{g}(x,t)-
\tilde{\alpha}(x,t)|dv|^2\big) &\text{a.e.\ }& \text{on}\ \Gamma_T,\\
v\big|_{t=0}&=0 &&\text{in}\ \overline{\Omega}
\end{align*}
satisfies, by \eqref{36} and \eqref{40}, the \textit{a~priori} estimate
$$
\|v\|_{{\mathcal V}_{2(n+2),2(n+1)}(\mathcal{Q}_T)} \le C_4
$$
with $C_4=\varkappa(C_3+M_0+{\eta}(M_0))$ independent of $\sigma$. This suffices to combine the Leray--Schauder theorem (see, e.g., \cite[Theorem 11.6]{GT01}) with Corollary~\ref{cor-uniqueness} in order to get unique solvability of \eqref{VP1-delta}--\eqref{VP3-delta} with $\delta=\varkappa$.

To complete the proof of Theorem~\ref{Du-est}, we take  successively $\delta_1=k\varkappa$, $\delta_2=(k+1)\varkappa$, $k\in\mathbb{N}$, and repeat the above procedure. Finitely many iterations of \eqref{39} lead to \eqref{grad-est}
since $v_1$ is nothing else than the solution $u$ of the problem \eqref{VP1}, \eqref{VP2} with homogeneous  initial condition.
\end{proof}

Based on the \textit{a~priori} gradient estimate derived in Theorem~\ref{Du-est}, we can get the solvability of the quasilinear parabolic Venttsel problem
\eqref{1.1-AN95}--\eqref{1.3-AN95}
under the hypotheses listed at the beginning of Section~\ref{sec4}.

\begin{theorem}
\label{quasilinear-existence}
Let $\partial\Omega \in \C^{1,1}$ and let the functions involved in \eqref{1.1-AN95}--\eqref{1.2-AN95} satisfy 
the conditions \eqref{A1}--\eqref{A5}, \eqref{V}, \eqref{V0}--\eqref{V5}.

If any solution $u\in V_{n+2,n+1}(\mathcal{Q}_T)$ to the one-parameter family of Venttsel problems
\begin{align}
\label{VP1-sigma}
\partial_tu -a^{ij}(x,t,u)D_{i}D_ju+ \sigma a(x,t,u,Du)=&\ 0 \quad \text{a.e. in}\ \mathcal{Q}_T,\\
\label{VP2-sigma}
\partial_tu -\alpha^{ij}(x,t,u)d_{i}d_ju+
\sigma\alpha(x,t,u,Du)
=&\ 0  \quad \text{a.e. on}\ \Gamma_T,\\
\label{VP3-sigma}
u\big|_{t=0}=&\ 0  \quad
\text{in}\  \overline{\Omega}
\end{align}
satisfies the \textit{a~priori} estimate
\begin{equation}\label{L-infty}
    \sup_{\mathcal{Q}_T} |u|\le M_0
\end{equation}
with $M_0$ independent of $u$ and $\sigma\in[0,1]$, then the quasilinear parabolic Venttsel problem \eqref{1.1-AN95}--\eqref{1.3-AN95} is solvable in the space $V_{n+2,n+1}(\mathcal{Q}_T)$.
\end{theorem}

\begin{proof}
We again proceed by using the Leray--Schauder theorem. Introduce the space
$$
\widetilde{\mathcal{V}}_{2(n+2),2(n+1)}(\mathcal{Q}_T):=\Big\{v\in \mathcal{V}_{2(n+2),2(n+1)}(\mathcal{Q}_T) \colon\quad \partial_{\mathbf{n}}v\in L^{1}(\Gamma_T)\Big\},
$$
equipped with the norm
$$
\|v\|_{\widetilde{\mathcal{V}}_{2(n+2),2(n+1)}(\mathcal{Q}_T)}=
\|v\|_{\mathcal{V}_{2(n+2),2(n+1)}(\mathcal{Q}_T)}+
\|\partial_{\mathbf{n}}v\|_{L^{1}(\Gamma_T)}.
$$
and define the nonlinear operator 
$$\mathcal{F}_1\colon\ \ \widetilde{\mathcal{V}}_{2(n+2),2(n+1)}(\mathcal{Q}_T) \mapsto V_{n+2,n+1}(\mathcal{Q}_T)
$$ 
which associates to any  $u\in \widetilde{\mathcal{V}}_{2(n+2),2(n+1)}(\mathcal{Q}_T)$ the solution $v=\mathcal{F}_1(u)$ of the \textit{linear} parabolic Venttsel problem
\begin{align}
\label{bb-L}
{\mathbb L} v:=
\partial_tv-\tilde{a}^{ij}D_iD_jv + \tilde{b}^iD_iv =&\  -\tilde{a}(\mu|Du|^2+\Phi) &\text{a.e. in}\ \mathcal{Q}_T,\\
\label{bb-B}
{\mathbb B}v:=
 \partial_tv-\tilde{\alpha}^{ij}d_id_jv +
\tilde\beta^{*i}d_iv\ +&\ \tilde\beta_0\partial_{\mathbf{n}}v \\
\nonumber
 =&\ -\tilde\alpha(\mu|du|^2+\Theta) &\text{a.e. on}\ \Gamma_T,\\
 \label{bb-I}
 v\big|_{t=0}=&\ 0  &\
\text{in}\ \ \overline{\Omega},
\end{align}
where $\tilde{a}^{ij}$, $\tilde{a}$ and $\tilde{b}^i$ are defined by  \eqref{tilde-aij}--\eqref{tilde-b}, while $\tilde{\alpha}^{ij}$, $\tilde{\alpha}$, $\tilde\beta^{*i}$ and $\tilde\beta_0$ are given by \eqref{tilde-alpij}--\eqref{beta0u}, respectively.

Similarly to the proof of Theorem \ref{Du-est}, we establish that the problem \eqref{bb-L}--\eqref{bb-B} satisfies all the hypotheses of Theorem~\ref{apriori-estimate}.  On the base of
Theorem \ref{existence} we conclude that the problem \eqref{bb-L}--\eqref{bb-I} is uniquely solvable in $V_{n+2,n+1}(\mathcal{Q}_T)$ and therefore the nonlinear operator $\mathcal{F}_1$ is well defined. Moreover, it is easy to see that the original quasilinear parabolic Venttsel problem \eqref{1.1-AN95}--\eqref{1.3-AN95} is equivalent to the equation $u=\mathcal{F}_1(u)$. 

The compactness of the embedding $V_{n+2,n+1}(\mathcal{Q}_T)\hookrightarrow\widetilde{\mathcal V}_{2(n+2),2(n+1)}(\mathcal{Q}_T)$ guarantees the compactness of the operator $\mathcal{F}_1$ considered as a mapping from $\widetilde{\mathcal V}_{2(n+2),2(n+1)}(\mathcal{Q}_T)$ into itself.

To prove the continuity of $\mathcal{F}_1$ in the space $\widetilde{\mathcal V}_{2(n+2),2(n+1)}(\mathcal{Q}_T)$,
we consider a sequence $u_h\in \widetilde{\mathcal V}_{2(n+2),2(n+1)}(\mathcal{Q}_T)$ such that $u_h\to u$ in $\widetilde{\mathcal V}_{2(n+2),2(n+1)}(\mathcal{Q}_T)$ as $h\to \infty$, and set $v_h=\mathcal{F}_1(u_h)$, $v=\mathcal{F}_1(u)$. Thus, $v$ is the solution of \eqref{bb-L}--\eqref{bb-I} while $v_h$ solves the problem
\begin{align}
\label{bb-Lh}
{\mathbb L}_h v_h:= \partial_tv_h-\tilde{a}_h^{ij}D_iD_jv_h\ +&\  \tilde{b}_h^iD_iv_h \\
\nonumber
=&\ -\tilde{a}_h(\mu|Du_h|^2+\Phi) &\text{a.e. in}\ \mathcal{Q}_T,\\
\label{bb-Bh}
{\mathbb B}_h v_h:= \partial_tv_h-\tilde{\alpha}_h^{ij}d_id_jv_h\ +&\ 
\tilde\beta^{*i}_hd_iv_h + \tilde\beta_{0,h}\partial_{\mathbf{n}}v_h\\ 
\nonumber
=&\ -\tilde\alpha_h(\mu|du_h|^2+\Theta) &\text{a.e. on}\ \Gamma_T,\\
\label{bb-Ih}
 v_h\big|_{t=0}=&\ 0  &\
\text{in}\ \ \overline{\Omega},
\end{align}
where $\tilde{a}_h^{ij}$, $\tilde{a}_h$, $\tilde{b}_h^i$,  $\tilde{\alpha}_h^{ij}$, $\tilde{\alpha}_h$, $\tilde\beta^{*i}_h$ and $\tilde\beta_{0,h}$ are defined similarly to \eqref{tilde-aij}--\eqref{tilde-b} and \eqref{tilde-alpij}--\eqref{beta0u} with $u$ replaced by $u_h$.

By
 the Arzel{\`a}-Ascoli Theorem, the set $\{u_h\}$ is uniformly equicontinuous. So, by Remark~\ref{same-way}, the
 $\textit{VMO}_x$-moduli of $\tilde{a}_h^{ij}(x,t)$  and $\tilde{\alpha}_h^{ij}(x,t)$ are uniformly bounded. Further, by definition we have
 \begin{gather*}
|\tilde{b}_h^i(x,t)|\le {\eta}(M)b(x,t),\qquad 
|\tilde\beta^{*i}_h(x,t)|\le {\eta}(M)\beta(x,t),\\ 
0\le\tilde\beta_{0,h}(x,t)\le{\eta}(M)\beta_0(x,t),
\end{gather*}
where $M=\sup_h\|u_h\|_{\mathcal{C}(\overline{\mathcal{Q}}_T)}$ while
the functions $b$, $\beta$ and $\beta_0$ verify the assumptions \eqref{L3}, \eqref{B3} and \eqref{B5}, respectively. It follows from 
Theorem~\ref{apriori-estimate} that 
\begin{equation*}
\|v_h\|_{V_{n+2,n+1}(\mathcal{Q}_T)} \le C_4{\eta}(M) \Big( \left\|\mu|Du_h|^2+\Phi\right\|_{n+2,\mathcal{Q}_T}
+\left\|\mu|du_h|^2+\Theta \right\|_{n+1,\Gamma_T} \Big), 
\end{equation*}
with constant $C_4$ independent of $h$ and $w_h$ (recall that $|\tilde a_h(x,t)|\le {\eta}(M)$ by \eqref{A4} and
$|\tilde \alpha_h(x,t)|\le {\eta}(M)$ by \eqref{V4}).
The  terms on the right-hand side above are uniformly bounded because of the boundedness of $u_h$ in $\widetilde{\mathcal V}_{2(n+2),2(n+1)}(\mathcal{Q}_T)$.
Thus, the sequence $v_h$ is bounded in $V_{n+2,n+1}(\mathcal{Q}_T)$. 

The difference $v-v_h$ solves the problem
\begin{align*}
{\mathbb L} (v-v_h)=&\ 
\big(\tilde{a}^{ij}-\tilde{a}_h^{ij}\big)D_iD_jv_h 
+ \big(\tilde{b}_h^i-\tilde{b}^i\big)D_iv_h \\
&\quad +  \tilde{a}_h\big(\mu|Du_h|^2+\Phi\big) - \tilde{a}\big(\mu|Du|^2+\Phi\big)\qquad\quad \text{a.e. in}\ \mathcal{Q}_T,
\\
{\mathbb B} (v-v_h)=&\ \big(\tilde{\alpha}^{ij}-\tilde{\alpha}_h^{ij}\big)
d_id_jv_h +
\big(\tilde\beta^{*i}_h-\tilde\beta^{*i}\big)d_iv_h +  \big(\tilde\beta_{0,h}-\tilde\beta_0\big)\partial_{\mathbf{n}}v_h \\ 
&\quad +  \tilde\alpha_h\big(\mu|du_h|^2+\Theta\big) - \tilde\alpha\big(\mu|du|^2+\Theta\big)
\qquad\qquad \text{a.e. on}\ \Gamma_T
\end{align*}
with the homogeneous initial data. The estimate \eqref{2.3-AN95} yields
\begin{align}
\label{2.3h-AN95}
\nonumber
\|v&-v_h\|_{V_{n+2,n+1}(\mathcal{Q}_T)}\\
\nonumber
& \le C_1\Big(\left\|\big(\tilde{a}^{ij}-\tilde{a}_h^{ij}\big)D_{i}D_jv_h  \right\|_{n+2,\mathcal{Q}_T}+  \left\|\big(\tilde{\alpha}^{ij}-\tilde{\alpha}_h^{ij}\big)
d_id_jv_h\right\|_{n+1,\Gamma_T}
\\
\nonumber
 & +
\big\|\big(\tilde{b}_h^i-\tilde{b}^i\big)d_iv_h\big\|_{n+2,\mathcal{Q}_T} + \big\|\big(\tilde\beta^{*i}_h-\tilde\beta^{*i}\big)d_iv_h\big\|_{n+1,\Gamma_T}  \\
 &+  \big\|\big(\tilde\beta_{0,h}-\tilde\beta_0\big)\partial_{\mathbf{n}}v_h\big\|_{n+1,\Gamma_T} \\ 
\nonumber
&+ \mu\left\|\tilde{a}_h|Du_h|^2\! -\! \tilde{a}|Du|^2\right\|_{n+2,\mathcal{Q}_T} + \mu\left\|\tilde\alpha_h|du_h|^2\! -\! \tilde\alpha|du|^2\right\|_{n+1,\Gamma_T}\\
\nonumber
 &+  \left\|\big(\tilde{a}_h-\tilde{a}\big)\Phi\right\|_{n+2,\mathcal{Q}_T} +\left\|\big(\tilde\alpha_h-\tilde\alpha\big)\Theta\right\|_{n+1,\Gamma_T}\Big) .
\end{align}

We know that $\|D_{i}D_jv_h\|_{n+2,\mathcal{Q}_T}$ are bounded, while 
$$
\left\|\tilde{a}^{ij}-\tilde{a}_h^{ij}\right\|_{\infty,\mathcal{Q}_T}=\left\|a^{ij}(\cdot,u)-a^{ij}(\cdot,u_h)\right\|_{\infty,\mathcal{Q}_T}\to0
$$
as $h\to\infty$, as consequence of $u_h\to u$ in $\mathcal{C}(\overline{\mathcal{Q}}_T)$ and the hypothesis \eqref{A3}. 

Therefore, the first term 
on the right-hand side of \eqref{2.3h-AN95} tends to $0$ as $h\to \infty$. The second term is managed in the same way.

Further on, 
$$
\tilde\beta^{*i}_h-\tilde\beta^{*i}=\beta\,\bigg(\dfrac{\alpha(\cdot,u_h,du_h)}{\mu|du_h|^2+\beta|du_h|+\Theta}\, \dfrac{d_iu_h}{|du_h|}-\dfrac{\alpha(\cdot,u,du)}{\mu|du|^2+\beta|du|+\Theta}\, \dfrac{d_iu}{|du|}\bigg)
$$
tends to zero almost everywhere on $\Gamma_T$ as $h\to\infty$. Due to the evident inequality $\big|\tilde\beta^{*i}_h(x,t)-\tilde\beta^{*i}(x,t)\big|\le 2{\eta}(M)\beta(x,t)$, the hypothesis \eqref{V5} together with the Lebesgue dominated convergence theorem ensure that
$$
\int\limits_{\Gamma_T}
\big|\tilde\beta^{*i}_h(x,t)-\tilde\beta^{*i}(x,t)\big|^{n+1}
\left(\log{\left(1+\big|\tilde\beta^{*i}_h(x,t)-\tilde\beta^{*i}(x,t)\big|\right)}\right)^{n}dxdt \to0
$$
as $h\to\infty$, that is
$$
\|\tilde{\boldsymbol{\beta}}{}^*_h-\tilde{\boldsymbol{\beta}}{}^*\|_{\mathcal{X}(\Gamma_T)} \to0,\qquad h\to\infty,
$$
where $\mathcal{X}(\Gamma_T)$ stands for the Orlicz space defined by the first relation in \eqref{V5}.

Theorem 10.5 in \cite{BIN75} shows that
$$
\big\|\big(\tilde\beta^{*i}_h-\tilde\beta^{*i}\big)d_iv_h\big\|_{n+1,\Gamma_T}\le 
\big\|\tilde{\boldsymbol{\beta}}{}^*_h-\tilde{\boldsymbol{\beta}}{}^*\big\|_{\mathcal{X}(\Gamma_T)}
\big\|v_h\big\|_{W^{2,1}_{n+1}(\Gamma_T)}\to 0 
$$
as $h\to\infty$.
The third and the fifth terms on the right-hand side of \eqref{2.3h-AN95} are managed in the same manner. Further on,
\begin{align*}
\left\|\tilde{a}_h|Du_h|^2 - \tilde{a}|Du|^2\right\|_{n+2,\mathcal{Q}_T}\le &\ 
\|\tilde{a}_h\|_{\infty,\mathcal{Q}_T}\left\||Du_h|^2 - |Du|^2\right\|_{n+2,\mathcal{Q}_T}\\
&\ + \left\|\big(\tilde{a}_h- \tilde{a}\big) |Du|^2\right\|_{n+2,\mathcal{Q}_T}.
\end{align*}
The first term here tends to zero since $u_h\to u$ in $\widetilde{\mathcal V}_{2(n+2),2(n+1)}(\mathcal{Q}_T)$, while the second one is infinitesimal by the Lebesgue theorem. All the remaining terms in \eqref{2.3h-AN95} are estimated in a similar way and we obtain finally $v_h\to v$ in $V_{n+2,n+1}(\mathcal{Q}_T)$  that proves the continuity of the operator $\mathcal{F}_1$.

\medskip

In order to apply the Leray--Schauder theorem and to get the existence of a fixed point of $\mathcal{F}_1$, we present a family of continuous, compact nonlinear operators $\mathcal{T}(\cdot, \sigma)$ continuously depending on the parameter $\sigma\in[0,1]$ such that 
$$
\mathcal{T}(u, 0)\equiv0;\qquad \mathcal{T}(u, 1)=\mathcal{F}_1(u), \qquad u\in \widetilde{\mathcal V}_{2(n+2),2(n+1)}(\mathcal{Q}_T).
$$
Namely, given a $\sigma\in[0,1]$, the operator $\mathcal{T}(\cdot, \sigma)$
associates to any  $u\in \widetilde{\mathcal V}_{2(n+2),2(n+1)}(\mathcal{Q}_T)$ the unique solution $v_\sigma$ of the linear parabolic Venttsel problem
\begin{align*}
&\partial_tv_{\sigma}-\tilde{a}^{ij}D_iD_jv_\sigma + \sigma\tilde{b}^iD_iv_\sigma = -\sigma\tilde{a}\big(\mu|Du|^2+\Phi\big)
&\text{a.e. in}\ \mathcal{Q}_T,\\
&\partial_tv_{\sigma}-\tilde{\alpha}^{ij}d_id_jv_\sigma +
\sigma\tilde\beta^{*i}d_iv_\sigma +  \sigma\tilde\beta_0\partial_{\mathbf{n}}v_\sigma
= -\sigma\tilde\alpha\big(\mu|du|^2+\Theta\big) 
&\text{a.e. on}\ \Gamma_T
\end{align*}
with the homogeneous initial data.

All mentioned properties of the family $\mathcal{T}(\cdot, \sigma)$ follow from the previous arguments, and to apply the Leray--Schauder theorem it remains only to derive the \textit{a~priori} estimate
\begin{equation}
\label{AA}
\|u\|_{\widetilde{\mathcal V}_{2(n+2),2(n+1)}(\mathcal{Q}_T)} \le C_5
\end{equation}
 for any solution of the equation $u =\mathcal{T}(u, \sigma)$ in $\widetilde{\mathcal V}_{2(n+2),2(n+1)}(\mathcal{Q}_T)$, with a constant $C_5$ independent of $\sigma\in[0,1]$ and $u$.

To this end, we notice that the equation $u =\mathcal{T}(u, \sigma)$
  is equivalent to the Venttsel problem \eqref{VP1-sigma}--\eqref{VP3-sigma}. We apply 
Theorem \ref{Du-est} and take into account that
under the assumption \eqref{L-infty} the constant $M_1$ in \eqref{grad-est} can be evidently chosen to be independent of $\sigma$. 

Finally, we rewrite the equations \eqref{VP1-sigma}--\eqref{VP2-sigma} in the form
\begin{align*}
&\partial_tu-\tilde{a}^{ij}D_iD_ju + \sigma\tilde{b}^iD_iu =  -\sigma F
&\text{a.e. in}\ \mathcal{Q}_T,\\
&\partial_tu-\tilde{\alpha}^{ij}d_id_ju +
\sigma\tilde\beta^{*i}d_iu + \sigma\tilde\beta_0\partial_{\mathbf{n}}u =
-\sigma G
&\text{a.e. on}\ \Gamma_T,
\end{align*}
with
\begin{align*}
F = \tilde{a}\big(\mu|Du|^2+\Phi\big),\qquad
G = \tilde\alpha\big(\mu|du|^2+\Theta\big),
\end{align*}
and conclude by \eqref{grad-est} that $\|F\|_{n+2,\mathcal{Q}_T}$ and $\|G\|_{n+1,\Gamma_T}$ are bounded uniformly with respect to $\sigma$ and $u$. Thus, \eqref{2.3-AN95} implies
$$
\|u\|_{V_{n+2,n+1}(\mathcal{Q}_T)}\le C
$$
with $C$ independent of $u$ and $\sigma$, which immediately implies the desired \textit{a~priori} estimate \eqref{AA}.
Application of the Leray--Schauder fixed point theorem completes the proof of 
Theorem~\ref{quasilinear-existence}. 
\end{proof}

Theorem~\ref{quasilinear-existence} is of conditional
 type. It reduces the solvability of the quasilinear Venttsel problem \eqref{1.1-AN95}--\eqref{1.3-AN95} to the \textit{a~priori} estimate \eqref{L-infty} for {\it any} solution $u\in V_{n+2,n+1}(\mathcal{Q}_T)$ of the problem \eqref{VP1-sigma}--\eqref{VP3-sigma}, with $M_0$ independent of $u$ and $\sigma\in[0,1]$. We provide now two simple sufficient conditions ensuring the validity of \eqref{L-infty}.\footnote{Similarly to the footnote \ref{foot}, the assumptions on $\partial\Omega$ here can be weakened.}

\begin{thm}\label{M0}
Let $\partial\Omega\in \mathcal{C}^{1,1}$. Assume that conditions \eqref{A1}, \eqref{V} and \eqref{V1} are satisfied, together with
 \eqref{A4} and \eqref{V4} where
$$
b, \Phi \in L^{n+1}(\mathcal{Q}_T);\quad
\beta,\Theta \in L^{n}(\Gamma_T).
$$

\begin{itemize}
    \item[{\bf 1.}] (\cite[Lemma 4.1]{AN96}). 
If $\eta(\zeta)\equiv1$ then any solution $u\in V_{n+1,n}(\mathcal{Q}_T)$ of  \eqref{VP1-sigma}--\eqref{VP3-sigma} satisfies the bound \eqref{L-infty}, where $M_0$ depends only on $n$, $\nu$, $\mu$, $\mathrm{diam}\,\Omega$, $T$,   $\|b\|_{n+1,\mathcal{Q}_T}$, $\|\Phi\|_{n+1,\mathcal{Q}_T}$, $\|\beta\|_{n,\Gamma_T}$,
$\|\Theta\|_{n,\Gamma_T}$, and on the properties of $\partial\Omega$.

\item[{\bf 2.}] 
Suppose that for $|z|\ge z_0$ the functions $a(x,t,z,p)$ and $\alpha(x,t,z,p^*)$ with $p^*\perp {\bf n}(x)$ are weakly differentiable with respect to $z$ and
$$
a_z(x,t,z,p)\ge \theta_0 \Phi(x,t);\quad
\alpha_z(x,t,z,p^*)\ge \theta_0 \Theta(x,t),\quad \theta_0={\text{\rm const}}>0.
$$

Then any solution $u\in V_{n+1,n}(\mathcal{Q}_T)$ of  \eqref{VP1-sigma}--\eqref{VP3-sigma} satisfies the bound \eqref{L-infty} with $M_0=\theta_0^{-1}\eta(z_0)+z_0$.
\end{itemize}

\end{thm}

\begin{proof}
{\bf 1.} It suffices to apply Proposition~\ref{global-max-principle} to  $v^{\pm}=\exp(\pm\mu\nu^{-1}u)-1$. 
\medskip

{\bf 2.} We will show the estimate $u\le M_0$, while the proof of $u\ge -M_0$ is similar.

Suppose that the set $\mathcal{Q}^+=\{u>z_0\}$ is non-empty. Then, using \eqref{A4} and the lower bound for $a_z$, we can write a.e. in $\mathcal{Q}_T\cap \mathcal{Q}^+$ that
\begin{align*}
&\ a(x,t,u,Du) =a(x,t,z_0,Du) \\
&\ +\int\limits_0^1 a_z\big(x,t,u(x,t)+(1-s)z_0,Du(x,t)\big)(u-z_0)\,ds \\
&\ \ge -\eta(z_0)\Big(\mu |Du(x,t)|^2+b(x,t)|Du(x,t)|+\Phi(x,t)\Big)+\theta_0 \Phi(x,t)(u-z_0).
\end{align*}
Therefore, \eqref{VP1-sigma} yields
\begin{align}
\label{m1}
\mathcal{L}u(x,t) :=&\ \partial_tu -\tilde{a}^{ij}(x,t)D_iD_ju +\sigma\eta(z_0) b^i(x,t)D_iu +\sigma\theta_0\Phi(x,t)u\\
\nonumber
\le&\ \big(\eta(z_0)+\theta_0z_0\big)\,\sigma\Phi(x,t)
= \mathcal{L}\big(\theta_0^{-1}\eta(z_0)+z_0\big),
\end{align}
where 
\begin{align*}
\tilde{a}^{ij}(x,t)&:={a}^{ij}\big(x,t,u(x,t)\big),\\ b^i(x,t)&:=-D_iu(x,t)\Big(\mu+\dfrac{b(x,t)}{|Du(x,t)|}\Big)\in L^{n+1}(\mathcal{Q}_T).
\end{align*}

Further on, we write $\alpha(x,t,u,Du)=\alpha(x,t,u,du)+\tilde\beta_0(x,t)\partial_{\mathbf{n}}u$ as in the proof of Theorem~\ref{Du-est} with $\tilde\beta_0(x,t)$ given by \eqref{beta0u}. Using \eqref{V4} and the lower bound for $\alpha_z$, we obtain a.e. on $\Gamma_T\cap\mathcal{Q}^+$ that
\begin{align*}
\alpha(x,t,u,Du)& =\alpha(x,t,z_0,du)\\
+&\ \int\limits_0^1 \alpha_z\big(x,t,u(x,t)+(1-s)z_0,du(x,t)\big)(u-z_0)\,ds +\tilde\beta_0(x,t)\partial_{\mathbf{n}}u \\
\ge&\  -\eta(z_0)\Big(\mu |du(x,t)|^2+\beta(x,t)|du(x,t)|+\Theta(x,t)\Big)\\
&\ +\theta_0 \Theta(x,t)(u-z_0)+\tilde\beta_0(x,t)\partial_{\mathbf{n}}u.
\end{align*}
This way, \eqref{VP2-sigma} implies
\begin{align}
\label{m2}
\mathcal{B}u(x,t)&\  :=\partial_tu-\tilde{\alpha}^{ij}(x,t)d_id_ju+\sigma\eta(z_0) \beta^{*i}(x,t)d_iu+\sigma \tilde\beta_0(x,t)\partial_{\mathbf{n}}u\\
\nonumber
&\  +\sigma\theta_0\Theta(x,t)u
\le \big(\eta(z_0)+\theta_0z_0\big)\,\sigma\Theta(x,t)\le \mathcal{B}\big(\theta_0^{-1}\eta(z_0)+z_0\big)
\end{align}
with 
\begin{align*}
\tilde{\alpha}^{ij}(x,t)&:={\alpha}^{ij}\big(x,t,u(x,t)\big),\\
\beta^{*i}(x,t)&:=-d_iu(x,t)\Big(\mu+\dfrac{\beta(x,t)}{|du(x,t)|}\Big)\in L^{n}(\Gamma_T).
\end{align*}
It is to be noted  that $\tilde{\beta}_0\ge0$ as it follows from \eqref{V}, while $\sigma\theta_0\Theta(x,t)\ge0$.

Setting $M_0:=\theta_0^{-1}\eta(z_0)+z_0$, we have from \eqref{m1} and \eqref{m2} that
$$
\CL(u-M_0)\le 0\quad \text{a.e. in}\ \mathcal{Q}_T\cap\mathcal{Q}^+,\qquad
\CB(u- M_0)\le 0\quad \text{a.e. on}\ \Gamma_T\cap\mathcal{Q}^+.
$$
Proposition~\ref{global-max-principle} provides $\sup_{\mathcal{Q}_T}\,(u-M_0)\le0$, that completes the proof.
\end{proof}


\section{Open problems}

A very interesting and important from the point of view of possible applications question is to extend the linear theory presented in Section~\ref{sec3} to the framework of \textit{anisotropic} Sobolev spaces where the exponents of Lebesgue integrability in $t$ and in $x$ are generally different. The PDEs background in such spaces is already available in \cite{Kr07b} and what is necessary to get is a fine extension result in the spirit of Theorem~\ref{extension-theorem} in the framework of anisotropic Sobolev spaces with precise control on the corresponding exponents.

 It would be interesting also to study linear and quasilinear parabolic two-phase Venttsel problems where the Venttsel condition is given on an interface separating $\Omega$ into two parts. For the case of continuous principal coefficients, this problem was investigated in
\cite{AN01}, \cite{AN00a}, and \cite{N04}. If the interface does not intersect the exterior boundary of $\Omega$ then our results can be transferred to the two-phase case more or less directly. However, if the interface meets $\partial\Omega$ transversally then both parts of $\Omega$ are domains with smooth closed edges. 
This requires to study the problem in weighted
Sobolev spaces where the weight is a power of the distance from a point to the edge (see  \cite{N01} and \cite{N04} for more details). 

The principal coefficients of the quasilinear operators considered in Section~\ref{sec4} depend on $(x,t)$ and the solution $u$, but these are independent of the gradient $Du$. An engaging open problem could be to prove strong solvability for  \textit{general quasilinear} Venttsel problems with principal parts depending also on the spatial and tangential gradients $Du$ and $du$, respectively. The approach to this task can be also based on the Leray--Schauder theorem employed in Section~\ref{sec4} but the \textit{a~priori} estimates needed are much more involved in this situation.

\subsection*{Funding}
The research of D.A. and A.N. (Theorems \ref{extension-theorem}, \ref{apriori-estimate}, \ref{quasilinear-existence}, \ref{M0}) is supported by the Russian Science Foundation grant 22-21-00027 fulfilled at St. Petersburg Department of the Steklov Mathematical Institute.

D.P. and L.S. are members of INdAM-GNAMPA. The  research of L.S. is partially supported by the INdAM-GNAMPA project ``Regularity of the solutions of boundary problems for elliptic and parabolic operators''. The work of D.P. 
was supported by the Italian Ministry of Education, University and Research
under the Programme ``Department of Excellence'' L. 232/2016 (Grant No. CUP -
D94I18000260001).

\bibliography{Bibliography-Venttsel-par}

\end{document}